%% file: ms.tex
\documentclass[12pt]{amsart}
\usepackage[english]{babel}
\usepackage[utf8]{inputenc}
\usepackage{amsmath, amsthm, amsfonts, amssymb}
\usepackage[colorlinks=true, linkcolor=red, citecolor=green, backref=page]{hyperref}
\usepackage{color}
\usepackage{mathtools}
\usepackage{graphicx}
\usepackage[english, plain]{fancyref}
\usepackage{centernot}
\usepackage{geometry}
\usepackage{enumitem}
\usepackage[makeroom]{cancel}

\input{definitions.tex}

\begin{document}

\title[Poles of the complex zeta function of a plane curve]{Poles of the complex zeta function of a plane curve}

\author[G. Blanco]{Guillem Blanco}

\thanks{ The author is supported by Spanish Ministerio de Economía y Competitividad MTM2015-69135-P and Generalitat de Catalunya 2017SGR-932 projects. }

\address{Departament de Matemàtiques\\
Univ. Politècnica de Catalunya\\
Av. Diagonal 647, Barcelona 08028, Spain.}
\email{Guillem.Blanco@upc.edu}

\begin{abstract}
We study the poles and residues of the complex zeta function \( f^s \) of a plane curve. We prove that most non-rupture divisors do not contribute to poles of \( f^s \) or roots of the Bernstein-Sato polynomial \( b_f(s) \) of \( f \). For plane branches we give an optimal set of candidates for the poles of \( f^s \) from the rupture divisors and the characteristic sequence of \( f \). We prove that for generic plane branches \( f_{gen} \) all the candidates are poles of \( f_{gen}^s \). As a consequence, we prove Yano's conjecture for any number of characteristic exponents if the eigenvalues of the monodromy of \( f \) are different.
\end{abstract}

\maketitle

\section{Introduction} \label{sec:introduction}

Let \( k \) be either \( \mathbb{R} \) or \( \mathbb{C} \) and let \( \varphi(x) \in C^\infty_c(k^n) \) be an infinitely many times differentiable function with compact support. Define the archimedean zeta function \( f^s \) of a non-constant polynomial \( f(x) \in k[x_1, \dots, x_n] \) as the distribution
\begin{equation} \label{eq:introduction-1}
\langle f^s, \varphi \rangle = \int_{k^n} |f(x)|^{\delta s} \varphi(x)\, dx,
\end{equation}
for \( s \in \mathbb{C} \), \( \textrm{Re}(s) > 0 \), where \( \delta = 1 \) if \( k = \mathbb{R} \) and \( \delta = 2 \) if \( k = \mathbb{C} \). In the 1954 edition of the International Congress of Mathematicians, I.~M.~Gel'fand \cite{gelfand-ICM54} posed the following problem: first, determine whether \( f^s \) is a meromorphic function of \( s \) with poles forming several arithmetic progressions; second, study the residues at those poles. The problem is solved for some specific polynomials having simple singularities in the book of Gel'fand and Shilov \cite{gelfand-shilov}, by regularizing the integral in \fref{eq:introduction-1}. It is not after Hironaka's resolution of singularities \cite{hironaka1, hironaka2}, that Bernstein and S.~I.~Gel'fand \cite{bernstein-gelfand}, and independently Atiyah \cite{atiyah}, give a positive answer to Gel'fand's first question. Both results use resolution of singularities to reduce the problem to the monomial case, already settled in \cite{gelfand-shilov}, and give a sequence of candidates poles for \( f^s \) from the resolution data.

\vskip 2mm

A different approach to the same problem is considered by Bernstein \cite{bernstein71, bernstein72}, who develops the theory of \( D \)-modules and proves the existence of the Bernstein-Sato polynomial \( b_f(s) \) of \( f \) and its functional equation,
\begin{equation} \label{eq:introduction-2}
P(s) \cdot f^{s+1} = b_f(s)\,f^s,
\end{equation}
with \( P(s) \in D[s] \) being a differential operator. The existence of the Bernstein-Sato polynomial in the local case is due to Björk \cite{bjork74}. The global \( b_f(s) \) is equal to the least common multiple of all the local Bernstein-Sato polynomials \( b_{f, p}(s), p \in k^n \), see \cite{narvaez91}. The Bernstein-Sato polynomial coincides with the \( b \)-function in the theory of prehomogeneous vectors spaces developed by Sato \cite{sato-shintani-90, sato80} in the 1960s, hence the name. The rationality of the roots of the Bernstein-Sato polynomial is established by Malgrange \cite{malgrange75} for isolated singularities and by Kashiwara \cite{kashiwara76} in general, using resolution of singularities. One verifies, using the functional equation in \fref{eq:introduction-2} and integration by parts, that the poles of \( f^s \) are among the rationals \( s = \alpha - k  \), with \( b_f(\alpha) = 0 \) and \( k \in \mathbb{Z}_{\geq 0} \). Loeser \cite{loeser85} shows the equality between both sets for reduced plane curves and isolated quasi-homogeneous singularities.

\vskip 2mm

In this paper we examine the original questions of Gel'fand and we use resolution of singularities to study the possible poles and the residues of the complex zeta function of general plane curves, generalizing the ideas and results of Lichtin in \cite{lichtin85, lichtin89}. The main results of this work are the following:
\begin{itemize}
\item For any candidate pole \( \sigma \) of \( f^s \), we give a formula for its residue expressed as an integral along the exceptional divisor associated to \( \sigma \), see \fref{prop:residue}.
\item In \fref{thm:residue-non-rupture}, we prove that most non-rupture divisors do not contribute to the poles of \( f^s \), and consequently to the roots of \( b_f(s) \).
\end{itemize}
This result answers, for reduced plane curves, a question raised by Kollár \cite{kollar97} on which exceptional divisors contribute to roots of the Bernstein-Sato polynomial. It is already well-known that, for plane curves, non-rupture divisors do not contribute to topological invariants such as the eigenvalues of the monodromy \cite{acampo75, neumann83}, the jumping numbers \cite{smith-thompson07}, or the poles of Igusa's local zeta functions \cite{loeser88}. For irreducible plane curves, we use Teissier's monomial curve \cite{teissier-appendix} associated to the semigroup of \( f \) to refine our previous results:
\begin{itemize}
\item In \fref{thm:plane-branch-candidates}, we obtain an optimal set of candidates for the poles of \( f^s \) in terms of the rupture divisors, the characteristic sequence, and the semigroup of \( f \).
\item From this, in \fref{thm:generic}, we prove that if \( f_{gen} \) is generic among all plane branches with fixed characteristic sequence (in the sense that the coefficients of a \( \mu \)-constant deformation are generic), all the candidates are indeed poles of \( f_{gen}^s \).
\item As a consequence, in \fref{cor:conjecture-yano}, we prove a conjecture of Yano \cite{yano82} about the \( b \)-exponents of a plane branch for any number of characteristic exponents under the assumption that the eigenvalues of the monodromy of \( f \) are pairwise different.
\end{itemize}

\vskip 2mm

Both the roots of \( b_f(s) \) and the poles of \( f^s \) are related to the local geometry of \( f \). By a series of results of Malgrange \cite{malgrange75, malgrange83}, first in the isolated singularity case and later in general, for every root \( \alpha \) of \( b_f(s) \), the value \( \exp{(2 \pi i \alpha)} \) is an eigenvalue of the local monodromy at some point of \( f^{-1}(0) \) and every eigenvalue is obtained in this way. For an isolated singularity these results imply that the degree of \( b_f(s) \) is at most the Milnor number. Therefore, in general, every pole \( \sigma \) of the archimedean zeta function \( f^s \) has that \( \exp(2 \pi i \sigma) \) is an eigenvalue of the monodromy at some point of \( f^{-1}(0) \). Barlet \cite{barlet84} proves that all eigenvalues are obtained in this way if \( k = \mathbb{C} \). Moreover, Barlet \cite{barlet86} shows that if the monodromy action in the \( q \)--th cohomology group of the Milnor fiber at some point of \( f^{-1}(0) \) has an eigenvalue \( \exp{(-2 \pi i \alpha)} \) with a \( k \times k \) Jordan block, then \( f^s \) has a pole at \( -q-\alpha, \alpha \in [0, 1) \), of order at least \( k \).

\vskip 2mm

A non-archimedean version of the zeta function of a polynomial \( f \) can be defined by a \( p \)-adic version of the integral in \fref{eq:introduction-1}. These zeta functions over the \( p \)-adic fields are usually called Igusa's local zeta functions. They were first studied by Igusa in \cite{igusa74, igusa75}, where he proves that they are rational functions. The theory of Igusa's zeta functions is vast and has many connection with Singularity and Number Theory. For instance, Igusa's local zeta functions are related to the number of solutions of \( f \) modulo \( p^m \). In \cite{igusa88}, Igusa conjectures that if \( s \) is a pole of a \( p \)-adic zeta function of \( f \), \( \textrm{Re}(s) \) is a root of \( b_{f}(s) \). This conjecture is proved by Loeser for plane curves \cite{loeser88} and for many non-degenerate singularities \cite{loeser90}. There is a version of this conjecture for the monodromy instead of the Bernstein-Sato polynomial. By the results of Malgrange, the former implies the later. Little more is known in this case, see the work of Artal Bartolo, Cassou-Noguès, Luengo and Melle Hernández \cite{ABCNLMH05} on quasi-ordinary singularities, Bories and Veys \cite{bories-veys16} on non-degenerate surface singularities, and Budur, Musta\c{t}\v{a} and Teitler \cite{BMT11} on hyperplane arrangements. These two conjectures imply the same sort of conjectures for the topological zeta function introduced by Denef and Loeser \cite{denef-loeser92}, for which more cases are known. The reader is referred to the classical reports of Denef \cite{denef91} and Igusa \cite{igusa96}, and the references therein for the concrete definitions, results and conjectures in the theory of Igusa's zeta functions. A survey of Meuser \cite{meuser16} includes the more recent developments.

\vskip 2mm

There are many singularity invariants related to the poles of \( f^s \) and the roots of \( b_f(s) \). The log-canonical threshold of \( f \) \cite{kollar97} is minus the largest pole of \( f^s \) and the largest root of \( b_f(s) \) and it sets the maximal region of holomorphy of the integral in \fref{eq:introduction-1}. For isolated singularities, it coincides with the complex singularity index, a concept that dates back to Arnold \cite{AGV88}. Using resolution of singularities, one defines the multiplier ideals and the associated jumping numbers \cite{ELSV04}. The log-canonical threshold appears as the smallest jumping number. The opposites in sign to the jumping numbers in \( (0, 1] \) are always roots of the Bernstein-Sato polynomial \cite{ELSV04}, see also \cite{budur-saito05}. The spectral numbers \cite{steenbrink77, steenbrink89} are a set of logarithms of the eigenvalues of the monodromy constructed using the mixed Hodge structure of the cohomology of the Milnor fiber. In the isolated singularity case the spectral and jumping numbers in \( (0, 1] \) coincide, for non-isolated singularities see \cite{budur03}. Furthermore, for isolated singularities the opposites in sign to the spectral numbers are poles of \( f^s \), \cite{loeser85}.

\vskip 2mm

There exist algorithms to compute the Bernstein-Sato polynomial \( b_f(s) \) for arbitrary polynomials \( f \), see \cite{oaku97, noro02, levandovskyy-jorge12}. However, even computationally, determining the roots of \( b_f(s) \) is a hard problem. For bounds on the roots and their multiplicities see the results of Saito \cite{saito94}. For candidates for the roots, we refer to the result of Kashiwara \cite{kashiwara76} and a refinement given by Lichtin \cite{lichtin89}. Very few formulas for the Bernstein-Sato polynomial are known. If \( f \) is smooth at \( p \in k^n \), then \( b_{f, p}(s) = s + 1 \). Hence, if \( f \) is everywhere smooth, then \( b_{f}(s) = s + 1 \) and the converse is also true, \cite{briancon-maisonobe96}. The monomial case is obtained by a straightforward computation. Yano worked out many interesting examples in \cite{yano78}. For isolated quasi-homogeneous singularities, see \cite{malgrange75, yano78}. For isolated semi-quasi-homogeneous singularities, see \cite{saito89, BGMM89}. The case of hyperplane arrangements has been studied by Walther \cite{walther05} and Saito \cite{saito16}.

\vskip 2mm

The case of plane curves has attracted a lot of attention, see \cite{yano82, kato1, kato2, cassou-nogues86, cassou-nogues87, cassou-nogues88, hertling-schtalke99, BMT07}. It is well-known that the roots of the Bernstein-Sato polynomial, hence the poles of \( f^s \), can change within a deformation with constant Milnor number of a plane curve, see, for instance, the examples in \cite{kato1, kato2}. This contrasts with the fact that the spectral numbers and, therefore, the monodromy eigenvalues and the jumping numbers, remain constant in a deformation with constant Milnor number, \cite{varchenko82}. For irreducible plane curve singularities, there is a conjecture of Yano \cite{yano82} asserting that if the plane branch is generic among those with fixed characteristic sequence, the so-called \( b \)-exponents, see definition in \fref{sec:yano-conjecture}, are constant and depend only on the characteristic sequence. In addition, Yano conjectures a closed formula for the \( b \)-exponents. This conjecture has been verified by Cassou-Noguès \cite{cassou-nogues88} for a single characteristic exponent. Recently, Artal Bartolo, Cassou-Noguès, Luengo and Melle Hernández \cite{ABCNLMH16} proved Yano's conjecture for plane branches with two characteristic exponents and monodromy with different eigenvalues.

\vskip 2mm

This paper is organized as follows. In \fref{sec:complex-powers}, we first review the classical results on the regularization and analytic continuation of \( f^s \) for a general complex polynomial \( f \). We then focus on the connection with the Bernstein-Sato polynomial and its basic properties, specially for isolated singularities. Yano's conjecture is presented at the end. In \fref{sec:semigroup-monomial-curve}, we present all the relevant results about the characteristic sequence and the semigroup of a plane branch. Then, we introduce Teissier's monomial curve and its deformations. Throughout \fref{sec:resolution-plane-curves} we review the results on resolution of singularities of plane curves that will be needed in the following sections. \fref{sec:poles-plane-curves} deals with the poles and the formulas for the residues along an exceptional divisor of the complex zeta function of a general plane curve, specially the case of non-rupture exceptional divisors. In the last section, we focus on irreducible plane curves and the optimal sequence of candidates coming from their rupture divisors. Finally, we work on the generic case and we prove Yano's conjecture.

\vskip 2mm

\textbf{Acknowledgments.} The author would like to thank his advisors, Maria Alberich-Carramiñana and Josep Àlvarez Montaner, for the fruitful discussions, the helpful comments and suggestions, and the constant support during the development of this work.

\section{Analytic continuation of complex powers} \label{sec:complex-powers}

In this section, we will review the basic results on regularization of complex powers appearing in the book of Gel'fand and Shilov \cite{gelfand-shilov}. We will see how resolution of singularities is used to construct the analytic continuation of the complex zeta function of an arbitrary polynomial \( f \in \mathbb{C}[z_1, \dots, z_n] \). The Bernstein-Sato polynomial \cite{bernstein72} is presented next and it is used to construct the analytic continuation of \( f^s \) in a different way. In order to state Yano's conjecture \cite{yano82}, we introduce the equivalent definition of the Bernstein-Sato polynomial for isolated singularities by Malgrange \cite{malgrange75} in terms of the Brieskorn lattice.

\subsection{Regularization of complex powers} \label{sec:regularization}

We will take the set of \emph{test functions of complex variable} as the set of smooth, compactly supported functions \(\varphi : \mathbb{C}^n \longrightarrow \mathbb{C}\). The space of such functions is denoted by \(C^{\infty}_c(\mathbb{C}^n)\). Alternatively, we can consider the larger space of test functions consisting of Schwartz functions. From the analytic continuation principle one deduces that there are no holomorphic compactly supported functions. Therefore, any \(\varphi \in C^{\infty}_c(\mathbb{C}^n)\) has an holomorphic an antiholomorphic part, i.e. \(\varphi = \varphi(z, \bar{z})\).

\vskip 2mm

Let \(f(z) \in \mathbb{C}[z_1, \dots, z_n]\) be a non-constant polynomial. We define a parametric family of \emph{distributions} of complex variable \(f^s : C^{\infty}_c(\mathbb{C}^n) \longrightarrow \mathbb{C}\) given by
\begin{equation} \label{eq:integral-equation}
\langle f^s, \varphi \rangle := \int_{\mathbb{C}^n} \varphi(z, \bar{z})|f(z)|^{2s} dz d\bar{z},
\end{equation}
which is well-defined for any \(s \in \mathbb{C}\) with \(\textrm{Re}(s) > 0\). The dependence of \(f^s\) on the parameter \(s\) is holomorphic as we can differentiate under the integral symbol to obtain another well-defined distribution, namely
\begin{equation*}
\frac{d}{ds} \langle f^s, \varphi \rangle = \int_{\mathbb{C}^n} \varphi(z, \bar{z})|f(z)|^{2s}\log|f(z)|^2 dz d\bar{z} =  \Big\langle \frac{df^s}{ds}, \varphi \Big\rangle, \quad \textrm{Re}(s) > 0.
\end{equation*}
The distribution \( f^s \) or the function \( \langle f^s, \varphi \rangle \) are usually called the \emph{complex zeta function} of \( f \). This name goes back to I.~M.~Gel'fand \cite{gelfand-ICM54}. In \cite{gelfand-shilov}, it is shown how one can obtain the analytic continuation of \( \langle f^s, \varphi \rangle \) by \emph{regularizing} the integral in \fref{eq:integral-equation}, for some classes of polynomials. The concept of \emph{regularization} is better understood after the following example.

\vskip 2mm

If one takes the function \( z^{-3/2}  \), in general, its integral \( \langle z^{-3/2}, \varphi \rangle \) against a test function \( \varphi(z, \bar{z}) \) will diverge. However, if \( \varphi(z, \bar{z}) \) vanishes at zero, the integral converges. Any distribution whose action on the elements of \( C_{c}^\infty(\mathbb{C}) \) vanishing at zero coincides with the action of \( z^{-3/2} \) is a regularization of \( z^{-3/2} \). The regularization of a function with algebraic singularities is unique up to functionals concentrated in the zero locus, see \cite[I.1.7]{gelfand-shilov}. For a fixed \( s \in \mathbb{C} \), the \emph{canonical regularization} of the function \( z^s \), in the sense that it is the most natural, is presented in the following proposition.

\begin{proposition}[{\cite[B1.2]{gelfand-shilov}}, Gel'fand-Shilov regularization] \label{prop:regularization}
For any \(m \in \mathbb{Z}_{\geq 0}\), the regularization of the distribution \( z^s : C_c^\infty(\mathbb{C}) \longrightarrow \mathbb{C} \) is given by
\begin{equation} \label{eq:regularization}
\begin{split}
\langle z^{s}, \varphi \rangle =  & \int_{|z| \leq 1} \Big[\varphi(z, \bar{z}) - \sum_{k + l = 0}^{m-1} \varphi^{(k, l)}(\boldsymbol{\emph{0}}, \boldsymbol{\emph{0}}) \frac{z^k \bar{z}^l}{k! l!} \Big]|z|^{2s} dz d\bar{z}\\
+ & \int_{|z|>1} \varphi(z, \bar{z}) |z|^{2s} dz d\bar{z} - 2\pi i \sum_{k = 0}^{m-1} \frac{\varphi^{(k, k)}(\boldsymbol{\emph{0}}, \boldsymbol{\emph{0}})}{(k!)^2(s + k + 1)}, \quad \emph{\textrm{Re}}(s) > -m - 1,
\end{split}
\end{equation}
where \(\varphi^{(i, j)} := {\partial^{i+j} \varphi}/{\partial z^i \partial \bar{z}^j}\). Hence, \(z^{s}\) has poles at \(s = -k-1\) for \(k \in \mathbb{Z}_{\geq 0}\) with residues
\begin{equation*}
\Res_{s = -k - 1} z^{s} = -\frac{2\pi i}{(k!)^2} \delta_0^{(k, k)},
\end{equation*}
where \(\delta_0^{(i, j)}\) are the distributional derivatives of the Dirac's delta function defined by \( \langle \delta_0^{(i, j)}, \varphi\rangle := (-1)^{i + j}\varphi^{(i,j)}(\boldsymbol{\emph{0}}, \boldsymbol{\emph{0}})\). Furthermore, in the strip \(-m -1 < \emph{\textrm{Re}}(s) < -m\), \fref{eq:regularization} reduces to
\begin{equation*}
\langle z^{s}, \varphi \rangle = \int_{\mathbb{C}} \Big[\varphi(z, \bar{z}) - \sum_{k + l = 0}^{m-1} \varphi^{(k, l)}(\boldsymbol{\emph{0}}, \boldsymbol{\emph{0}}) \frac{z^k \bar{z}^l}{k! l!} \Big] |z|^{2s} dz d\bar{z}.
\end{equation*}
\end{proposition}

For a fixed \( \varphi \in C_c^\infty(\mathbb{C}) \), \fref{prop:regularization} gives the meromorphic continuation to the whole complex plane of the holomorphic function of \( s \) defined by the integral \( \langle z^s, \varphi \rangle \). For any polynomial \( f \), we will talk indistinguishably about the meromorphic continuation or the (canonical) regularization of its complex zeta function \( f^s \).

\begin{remark}
Although in \fref{prop:regularization} the test function \( \varphi \) is assumed to be in \(C^\infty_c(\mathbb{C}) \), the proof of the result only uses the fact that \( \varphi \) is infinitely differentiable near \( {0} \) and compactly supported. This means that the same result works for a meromorphic \( \varphi \) with poles away from \( {0} \) and compact support. In particular, if \( \varphi(z, \bar{z}; s) \in C^\infty(U) \), where \( U \) is a neighborhood of \( {0} \), and compactly supported, the poles of \( \langle z^s, \varphi(z, \bar{z}; s) \rangle \) will be the negative integers \(\mathbb{Z}_{<0} \) together with the poles of \( \varphi(z, \bar{z}; s) \) in \( s \) away from \( U \).
\end{remark}

Resolution of the singularities in \( f \in \mathbb{C}[z_1, \dots, z_n] \) is used in \cite{bernstein-gelfand} and \cite{atiyah} to reduce the problem of finding the analytic continuation of \( f^s \) to the monomial case considered in \fref{prop:regularization}. Let \( \pi : X' \longrightarrow \mathbb{C}^n \) be a resolution of \( f \) with \( F_\pi := \sum N_i D_i \) the total transform divisor and \( K_\pi := \sum k_i E_i \) the relative canonical divisor, and suppose that \( E := \textrm{Exc}(\pi) = \sum E_i \) is the exceptional locus. Let \( \{U_{\alpha}\}_{\alpha \in \Lambda} \) be an affine open cover of \( X' \). Take \( \{\eta_\alpha\} \) a partition of unity subordinated to the cover \( \{U_{\alpha}\}_{\alpha \in \Lambda} \). That is, \( \eta_\alpha \in C^\infty(\mathbb{C}^n) \) (not necessarily with compact support), \( \sum \eta_\alpha \equiv 1 \), with only finitely many \( \eta_\alpha \) being non-zero at a point of \( X' \) and \(\textrm{Supp}(\eta_\alpha) \subseteq U_{\alpha} \). Then, with a small abuse of notation,
\begin{equation} \label{eq:anal-cont0}
\begin{split}
  \langle f^s, \varphi \rangle & = \int_{X'} |\pi^* f|^{2s} (\pi^* \varphi) |d \pi |^2 \\
  & = \sum_{\alpha \in \Lambda} \int_{U_{\alpha}} |z_{1}|^{2(N_{1,\alpha} s + k_{1,\alpha})} \cdots |z_{n}|^{2(N_{n,\alpha} s + k_{n,\alpha})} |u_{\alpha}(z)|^{2s} |v_{\alpha}(z)|^2 \varphi_\alpha(z, \bar{z})\, dz d\bar{z},
\end{split}
\end{equation}
where \( \varphi_\alpha := \eta_\alpha \pi^* \varphi \) for each \( \alpha \in \Lambda \) and \( u_{\alpha}(\boldsymbol{0}), v_{\alpha}(\boldsymbol{0}) \neq 0 \). The resolution morphism \(\pi\) being proper implies that both \( E \) and \( \pi^{-1}(\textrm{Supp}(\varphi)) \) are compact sets. Since the singularities of the integral \( \langle f^s, \varphi \rangle \) are produced by the zero set of \( f \), in order to study the poles of \( f^s \), it is enough to consider a finite affine open cover \( \{U_{\alpha}\}_{\alpha \in \Lambda} \) of \( E \) consisting of neighborhoods of points \( p_\alpha \in E \) and such that \( \textrm{Supp}(\varphi) \subseteq \pi(\cup_\alpha U_{\alpha}) \).

\vskip 2mm

From \fref{eq:anal-cont0} and \fref{prop:regularization}, we see that each divisor \( D_i \) in the support of \( F_\pi \) generates a set of candidate poles of \( f^s \), namely
\begin{equation*}
-\frac{k_i + 1 + \nu}{N_i}, \quad \nu \in \mathbb{Z}_{\geq 0}.
\end{equation*}
The opposite in sign to the largest pole is the \emph{log-canonical threshold} \( \textrm{lct}(f) \) of \( f \) and sets the maximal region of holomorphy of \( \langle f^s, \varphi \rangle \) for a general \( \varphi \in C_c^\infty(\mathbb{C}^n) \). This solves Gel'fand's first question in \cite{gelfand-ICM54}. However, nothing is said about the residues of \( f^s \) at those poles. Moreover, the set of candidates is usually large compared with the set of roots of the Bernstein-Sato polynomial or the actual poles of \( f^s \).

\subsection{The Bernstein-Sato polynomial} \label{sec:bernstein-sato}

Let \(X\) be a complex manifold of dimension \(n\) with \(\mathcal{O}_X\) the sheaf of regular functions and \(\mathcal{D}_X\) the sheaf of differentials operators. We set \(\mathcal{D}_X[s] = \mathcal{D}_X \otimes_{\mathbb{C}} \mathbb{C}[s]\), where \(s\) is an indeterminate commuting with all differential operators. Fix \(f \in \mathcal{O}_X\) a non-zero regular function on \(X\). The Bernstein-Sato functional equation \cite{bernstein72} asserts the existence of a differential operator \(P(s) \in \mathcal{D}_X[s]\) and a non-zero polynomial \(b(s) \in \mathbb{C}[s]\) such that
\begin{equation} \label{eq:functional-equation}
P(s) \cdot f^{s+1} = b(s) f^s.
\end{equation}
Although neither \( P(s) \) or \( b(s) \) are necessarily unique, all the polynomials \( b(s) \in \mathbb{C}[s] \) satisfying \fref{eq:functional-equation} form an ideal. The unique monic generator of this ideal is called \emph{the Bernstein-Sato polynomial} \( b_f(s) \) of \( f \). Everything remains true in the local case, see \cite{bjork-book}. The only two general results about the structure of the roots of the Bernstein-Sato polynomial are the following.

\begin{theorem} [\cite{kashiwara76}, \cite{lichtin89}, Rationality of the roots] \label{thm:rationality}
Let \(f \in \mathcal{O}_{X}\) be non-constant and let \(\pi : X' \longrightarrow X\) be a resolution of \(f\) with \(F_\pi = \sum N_i E_i\) and \(K_\pi = \sum k_i E_i\) the resolution and relative canonical divisors. Then, the roots of \( b_f(s) \) are among the numbers
\begin{equation*}
  -\frac{k_i + 1 + \nu}{N_i} - k,
\end{equation*}
for \(\nu \in \{0, \dots, N_i-1\}\) and \(k \in \mathbb{Z}_{\geq 0}\). Therefore, all the roots are negative rational numbers.
\end{theorem}

As in the case of the candidate poles of the complex zeta function \( f^s \), the set of candidates provided by \fref{thm:rationality} is usually larger than the set of roots of \( b_f(s) \). Indeed, there are many exceptional divisors of the resolution not contributing to actual poles of \( f^s \) or roots of \( b_f(s) \). Since \( s = -1 \) is always a root of \( b_f(s) \), it is sometimes useful to work with the \emph{reduced} Bernstein-Sato polynomial, \( \widetilde{b}_f(s) := b_f(s)/(s+1) \).

\begin{theorem} [\cite{saito94}] \label{thm:saito}
Let \(f \in \mathcal{O}_{X}\) be non-constant. Then, for any root \( \alpha \) of \( \widetilde{b}_f(s) \), \( \alpha \in [-n + \emph{\textrm{lct}}(f), -\emph{\textrm{lct}}(f)]\). If \(m_{\alpha} \) denotes the multiplicity of \(\alpha\), \(m_{\alpha} \leq n - \emph{\textrm{lct}}(f) - \alpha + 1\).
\end{theorem}

The bounds in \fref{thm:saito} are optimal for isolated quasi-homogeneous singularities.

\vskip 2mm

The Bernstein-Sato functional equation together with integration by parts can be used to obtain the analytic continuation of \( f^s \) in a different way.

\begin{proposition} \label{prop:anal-continuation-bernstein}
The complex zeta function \(f^s\) admits a meromorphic continuation to \( \mathbb{C} \) with poles among the rational numbers \(\alpha - k \) with \( b_f(\alpha) = 0 \) and \( k \in \mathbb{Z}_{\geq 0} \).
\end{proposition}
\begin{proof}
We can use the functional \fref{eq:functional-equation} and integration by parts to analytically continue \fref{eq:integral-equation} in the following way
\begin{equation} \label{eq:anal-cont1}
\begin{split}
 \langle f^s, \varphi \rangle = \int_{\mathbb{C}^n} \varphi(z, \bar{z})|f(z)|^{2s} dzd\bar{z} & = \frac{1}{b^2_{f}(s)} \int_{\mathbb{C}^n} \varphi(z, \bar{z}) \big[P(s) \cdot f^{s+1}(z)\big]\big[\overbar{P}(s) \cdot f^{s+1}(\bar{z})\big] dz d\bar{z} \\
 & = \frac{1}{b^2_{f}(s)} \int_{\mathbb{C}^n} \overbar{P}^*P^*(s)\big(\varphi(z, \bar{z})\big)|f(z)|^{2(s+1)} dz d\bar{z}.
\end{split}
\end{equation}
The last term of \fref{eq:anal-cont1} defines an analytic function whenever \(\textrm{Re}(s) > -1\), except for possible poles at \(b_{f}^{-1}(0)\), and it is equal to \(\langle f^s, \varphi \rangle\) in \(\textrm{Re}(s) > 0\).
If a differential operator has the form \(P(s) = \sum_{\beta} a_{\beta}(s, z) (\frac{\partial}{\partial z})^\beta\), we have considered
\begin{equation*}
\overbar{P}(s) := \sum_{\beta} a_{\beta}(\bar{s}, \bar{z}) \Big(\frac{\partial}{\partial \bar{z}}\Big)^\beta, \quad P^*(s) := \sum_{\beta} (-1)^{|\beta|} \Big(\frac{\partial}{\partial z}\Big)^\beta a_{\beta}(z, s),
\end{equation*}
the \emph{conjugate} and \emph{adjoint operator} of \(P(s)\), respectively. Iterating the process we get
\begin{equation} \label{eq:anal-cont2}
\langle f^s, \varphi \rangle = \frac{\langle f^{s+k+1}, \overbar{P}^*P^*(s+k)\cdots \overbar{P}^*P^*(s)(\varphi) \rangle}{b^2_{f}(s) \cdots b^2_{f}(s+k)}, \quad {\textrm{Re}}(s) > -k -1,
\end{equation}
and the result follows.
\end{proof}

The set of poles of the complex zeta function \( f^s \) is known to be exactly the set \( \alpha - k \) with \( b_f(\alpha) = 0 \) and \( k \in \mathbb{Z}_{\geq 0} \) for reduced plane curve singularities and isolated quasi-homogeneous singularities, see \cite[Th. 1.9]{loeser85}. Therefore, at least in these cases, the divisors contributing to poles of the complex zeta function \( f^s \) are the same divisors that contribute to roots of the Bernstein-Sato polynomial \( b_f(s) \). However, even in these cases, it is not straightforward to relate poles of \( f^s \) with roots of \( b_f(s) \). In general, from \fref{thm:saito} and \fref{prop:anal-continuation-bernstein}, one has that,

\begin{corollary} \label{cor:shift-roots}
Every pole \(\sigma \in [-n+\emph{\textrm{lct}}(f), -\emph{\textrm{lct}}(f)]\) of $f^s$ such that \( \sigma + k \) is not a root of \( b_f(s) \) for all \( k \in \mathbb{Z}_{>0} \) is a root of \( b_{f}(s) \).
\end{corollary}

\subsection{Yano's conjecture} \label{sec:yano-conjecture}

We assume now that \( f \) has an isolated singularity at \( \boldsymbol{0} \) and we work locally around this point. Let \( \mu := \dim_{\mathbb{C}} \mathcal{O}_{X, \boldsymbol{0}}/ \langle \partial f/ \partial z_1, \dots, \partial f/\partial z_n \rangle \) be the \emph{Milnor number} of \( f \) at \( \boldsymbol{0} \). The Brieskorn lattice, \cite{brieskorn70} \( H''_{f, \boldsymbol{0}} := \Omega^n_{X, \boldsymbol{0}} / \dd f \wedge \dd \Omega^{n-1}_{X, \boldsymbol{0}} \), of \( f \) at \( \boldsymbol{0} \) has a structure of \( \mathbb{C}\{t\} \)-module given by the multiplication by \( f \), moreover, it is a free module of dimension \( \mu \) and carries a connection \( \partial_t \). Malgrange shows in \cite{malgrange75} that if \( \widetilde{H}''_{f, \boldsymbol{0}} := \sum_{k \geq 0} ( \partial_t t )^k H''_{f, \boldsymbol{0}} \) is the saturation of the Brieskorn lattice, then \( \widetilde{b}_{f, \boldsymbol{0}}(s) \) is the minimal polynomial of the endomorphism
\begin{equation*}
- \partial_t t : \widetilde{H}''_{f, \boldsymbol{0}}/ t \widetilde{H}''_{f, \boldsymbol{0}} \longrightarrow \widetilde{H}''_{f, \boldsymbol{0}} / t \widetilde{H}''_{f, \boldsymbol{0}}.
\end{equation*}
Therefore, the degree of the Bernstein-Sato polynomial of an isolated singularity is at most \( \mu \). If \( \widetilde{\alpha}_1, \dots, \widetilde{\alpha}_r, r \leq \mu \) are the roots of \( \widetilde{b}_{f, \boldsymbol{0}} \), Malgrange also proves that the polynomial \( \prod_j \big(s - \exp({2 \pi i \widetilde{\alpha}_j})\big)  \) is a divisor of the characteristic polynomial of the monodromy and a multiple of the minimal polynomial of the monodromy. Following the terminology of Yano, the \emph{\( b \)-exponents} of an isolated singularity \( f \) are the \( \mu \) roots \( {\alpha}_1, \dots, {\alpha}_\mu \) of the characteristic polynomial of the action of \( \partial_t t \) on \( \widetilde{H}''_{f, \boldsymbol{0}}/ t \widetilde{H}''_{f, \boldsymbol{0}} \).

\vskip 2mm

Yano's conjecture \cite{yano82} on the \( b \)-exponents of generic irreducible plane curves reads as follows. Let \( (n, \beta_1, \dots, \beta_g) \) be the characteristic sequence of a plane branch with \( g \geq 1 \) being the number of characteristic exponents. With the same notation as in \cite{yano82} define,
\begin{equation} \label{eq:yano-definitions}
\begin{split}
& e_{0} := n, \quad e_{i} := \gcd(n, \beta_1, \dots, \beta_i), \qquad i = 1, \dots, g, \\
& r_i := \frac{\beta_i + n}{e_{i}}, \quad R_i := \frac{\beta_i e_{i-1} + \beta_{i-1}(e_{i-2} - e_{i-1}) + \cdots + \beta_1(e_{0} - e_{1})}{e_{i}}, \\
& r'_0 := 2, \quad r'_i := r_{i-1} + \left\lfloor\frac{\beta_i - \beta_{i-1}}{e_{i-1}}\right\rfloor + 1 = \left\lfloor\frac{r_i e_{i}}{e_{i-1}}\right\rfloor + 1, \\
& R'_0 := n, \quad R'_i := R_{i-1} + \beta_i - \beta_{i-1} = \frac{R_i e_{i}}{e_{i-1}}.
\end{split}
\end{equation}
Inspired by A'Campo's formula \cite{acampo75} for the eigenvalues of the monodromy, Yano defines the following polynomial with fractional powers in \( t \)
\begin{equation} \label{eq:yano-generating-series}
R\big((n, \beta_1, \dots, \beta_g), t\big) := \sum_{i = 1}^g t^{\frac{r_i}{R_i}} \frac{1 - t}{1 - t^{\frac{1}{R_i}}} - \sum_{i = 0}^g t^{\frac{r'_i}{R'_i}} \frac{1 - t}{1 - t^{\frac{1}{R'_i}}} + t,
\end{equation}
and proves that \( R\big((n, \beta_1, \dots, \beta_n), t\big) \) has non-negative coefficients. Finally,

\begin{conjecture}[Yano, \cite{yano82}]
For generic curves among all irreducible plane curves with characteristic sequence \( (n, \beta_1, \dots, \beta_g) \) the \( b \)-exponents \( \alpha_1, \dots, \alpha_\mu \) are given by the generating function \( R\big((n, \beta_1, \dots, \beta_n), t\big) \). That is,
\begin{equation*}
\sum_{i = 1}^\mu t^{\alpha_i} = R\big((n, \beta_1, \dots, \beta_g), t\big).
\end{equation*}
\end{conjecture}

\section{The semigroup of a plane branch and its monomial curve} \label{sec:semigroup-monomial-curve}

The characteristic sequence of an irreducible plane curve is a complete topological invariant. Two germs of a curve are topologically equivalent, if and only if, they have the same characteristic sequence. Alternatively, one defines the semigroup of a plane branch from its associated valuation. The aim of this section is to first describe all the terminology related to the characteristic sequence and the semigroup of a plane branch and present the relation between them. Secondly, we will introduce Teissier's monomial curve associated to a semigroup and its deformations. The main references for this section are Zariski's book \cite{zariski-moduli} and its appendix \cite{teissier-appendix} by Teissier.

\subsection{The semigroup of a plane branch} \label{sec:semigroup}

We begin by fixing a germ of a plane branch \( f : (\mathbb{C}^2, \boldsymbol{0}) \longrightarrow (\mathbb{C}, 0) \) with \emph{characteristic sequence} \((n, \beta_1, \dots, \beta_g), n, \beta_i \in \mathbb{Z}_{>0} \). The characteristic sequence can be obtained from the Puiseux parameterization of \( f \). After an analytic change of variables we can always assume that \( n < \beta_1 < \cdots < \beta_g \). Define the integers \( e_i := \gcd(e_{i-1}, \beta_{i}), e_0 := n \), with \( n \) being the multiplicity of \( f \). Notice that they satisfy \( e_0 > e_1 > \cdots > e_g = 1 \) and \( e_{i-1} \centernot\mid \beta_i \).

\vskip 2mm

We set \( n_i := e_{i-1} / e_i \) for \( i = 1, \dots, g \) and, by convention, \( \beta_{0} := 0 \) and \( n_0 := 0 \). The integers \( n_1, \dots, n_g \) are strictly larger than 1 and we have that \( e_{i-1} = n_i n_{i+1} \cdots n_g \) for \( i = 1, \dots, g \). In particular, \( n = n_1 \cdots n_g \). The fractions \( m_i / n_1 \cdots n_{i} \), with \( m_i \) defined as \( m_i := \beta_{i} / e_i \), are the reduced \emph{characteristic exponents} appearing in the Puiseux series of \( f \). The tuples \( (m_i, n_i) \) satisfy \( \gcd(m_i, n_i) = 1 \) and are usually called the \emph{Puiseux pairs}.

\vskip 2mm

Let \( \mathcal{O}_f \) be the local ring of \( f \). The Puiseux parameterization of \( f \) gives an injection \( \mathcal{O}_{f} \xhookrightarrow{\quad} \mathbb{C}\{t\} \). We denote the \( t \)-adic valuation of \( \mathcal{O}_f \) by \( \nu \) and \( \Gamma \subseteq \mathbb{Z}_{\geq 0} \) denotes the associated semigroup
\[ \Gamma := \big\{\nu(g) \in \mathbb{Z}_{\geq 0} \ |\ g \in \mathcal{O}_{f} \setminus \{0\}\big\}. \]
Since \( f \) is irreducible there exists a minimum integer \( c \in \mathbb{Z}_{>0} \), the \emph{conductor} of \( \Gamma \), such that \( (t^c) \cdot \mathbb{C}\{t\} \subseteq \mathcal{O}_{f} \). As a result, any integer in \( [c, \infty) \) must belong to \( \Gamma \), which implies that \( \mathbb{Z}_{\geq 0} \setminus \Gamma \) is finite. Since \( \mathbb{Z}_{\geq 0} \setminus \Gamma \) is finite, we can find a minimal generating set \( \langle \overline{\beta}_{0}, \overline{\beta}_{1}, \dots, \overline{\beta}_{g} \rangle \) of \( \Gamma \), i.e. \( \overline{\beta}_{i} \) are the minimal integers such that \( \overline{\beta}_{i} \not\in \langle \overline{\beta}_0, \overline{\beta}_{1}, \dots, \overline{\beta}_{i-1} \rangle \), with \( \overline{\beta}_{0} < \overline{\beta}_{1} < \cdots < \overline{\beta}_{g} \) and \( \gcd(\overline{\beta}_{0} , \overline{\beta}_{1}, \dots, \overline{\beta}_{g}) = 1 \).

\vskip 2mm

The semigroup generators can be computed from the characteristic sequence in the following way, see \cite[II.3]{zariski-moduli},
\begin{equation} \label{eq:generators-semigroup1}
\overline{\beta}_{i} = (n_1 - 1)n_2 \cdots n_{i-1} \beta_{1} + (n_2- 1)n_3 \cdots n_{i-1} \beta_{2} + \cdots + (n_{i-1} - 1) \beta_{i-1} + \beta_{i},
\end{equation}
for \( i = 2, \dots, g \) and with  \( \overline{\beta}_{0} = n, \overline{\beta}_{1} = \beta_{1} \). Recursively this can be expressed as
\begin{equation} \label{eq:generators-semigroup2}
\overline{\beta}_{i} = n_{i-1} \overline{\beta}_{i-1} - \beta_{i-1} + \beta_{i}, \qquad i = 2, \dots, g.
\end{equation}
By \fref{eq:generators-semigroup1}, \( \gcd(e_{i-1}, \overline{\beta}_{i}) = e_i \) with \( e_0 = \overline{\beta}_{0} = n \) and \( e_{i-1} \centernot\mid \overline{\beta}_{i} \). In the same way as before, we define the sequence of integers \( \overline{m}_{i} := \overline{\beta}_{i} / e_i \) which will be useful in the sequel. The integers \( \overline{m}_i, i = 1, \dots, g \) can be obtained recursively using \fref{eq:generators-semigroup2}, namely
\begin{equation} \label{eq:reduced-generators-semigroup}
\overline{m}_{i} = n_i n_{i-1} \overline{m}_{i-1} - n_i m_{i-1} + m_i, \quad i = 2, \dots, g,
\end{equation}
with \( \overline{m}_0 = 1, \overline{m}_1 = m_1 \). Note that \fref{eq:reduced-generators-semigroup} implies \( \gcd(\overline{m}_{i} , n_i) = 1 \) for \( i = 1, \dots, g \). Finally, we define \( q_i := (\beta_{i} - \beta_{i-1})/e_i = m_i - n_i m_{i-1} \) for \( i = 1, \dots, g \). Alternatively, by \fref{eq:generators-semigroup2}, these quantities are \( q_i = (\overline{\beta}_{i} - n_{i-1}\overline{\beta}_{i-1})/e_i = \overline{m}_i - n_i n_{i-1} \overline{m}_{i-1} \) for \( i = 2, \dots, g \) and \( q_1 = \overline{m}_{1} = m_1 \).

\vskip 2mm

The following lemma is a fundamental property of the semigroups coming from plane branches.
\begin{lemma}[{\cite[I.2]{teissier-appendix}}] \label{lemma:semigroup}
If \( \Gamma = \langle \overline{\beta}_{0}, \overline{\beta}_{1}, \dots, \overline{\beta}_{g} \rangle \) is the semigroup of a plane branch, then
\[ n_i \overline{\beta}_i \in \langle \overline{\beta}_{0}, \overline{\beta}_1, \dots, \overline{\beta}_{i-1} \rangle. \]
\end{lemma}
The property in \fref{lemma:semigroup} together with the fact that \( n_i \overline{\beta}_{i} < \overline{\beta}_{i+1} \), which follows directly from \( \beta_i < \beta_{i+1} \) and \fref{eq:generators-semigroup2}, characterize the semigroups of plane branches.
\begin{proposition}[{\cite[I.3.2]{teissier-appendix}}] \label{prop:characterization}
\( \Gamma = \langle \overline{\beta}_{0}, \overline{\beta}_{1}, \dots, \overline{\beta}_{g}\rangle \subseteq \mathbb{Z}_{\geq 0} \) is the semigroup of a plane branch, if and only if, \( \gcd(\overline{\beta}_{0}, \overline{\beta}_{1}, \dots, \overline{\beta}_{g}) = 1 \),
\( n_i \overline{\beta}_i \in \langle \overline{\beta}_{0}, \overline{\beta}_1, \dots, \overline{\beta}_{i-1} \rangle \) for \( i = 1, \dots, g \), and \( n_i \overline{\beta}_{i} < \overline{\beta}_{i+1} \), for \( 1 \leq i < g \).
\end{proposition}

The conductor \( c \) of \( \Gamma \) can be computed as \( c = n_g \overline{\beta}_{g} - \beta_{g} - (n - 1) \), see \cite[II.3]{zariski-moduli}. Combining this formula with \fref{eq:generators-semigroup2}, we get the following formula for the conductor
\begin{equation} \label{eq:conductor2}
c = \sum_{i = 1}^{g} (n_i - 1) \overline{\beta}_{i} - n + 1.
\end{equation}
Finally, we would like to notice that the Milnor number, \( \mu \hspace{-1pt} = \hspace{-1pt} \dim_{\mathbb{C}} \hspace{-1pt} \mathbb{C}\{x, \hspace{-1pt} y\} / \langle \partial f\hspace{-0.5pt}/\hspace{-0.5pt}\partial x, \partial f\hspace{-0.5pt}/\hspace{-0.5pt}\partial y \rangle \), of the branch \( f \) coincides with its conductor \( c \), see \cite[6.4]{casas}. Therefore, \( \mu \) can be computed from a semigroup \( \Gamma = \langle \overline{\beta}_0, \overline{\beta}_1, \dots, \overline{\beta}_g \rangle \) using \fref{eq:conductor2}.

\subsection{The monomial curve and its deformations} \label{sec:monomial-curve}

Let \( \Gamma = \langle \overline{\beta}_0, \overline{\beta}_1, \dots, \overline{\beta}_g \rangle \subseteq \mathbb{Z}_{\geq 0}\) be a semigroup such that \( \mathbb{Z}_{\geq 0} \setminus \Gamma \) is finite, that is \( \gcd(\overline{\beta}_0, \dots, \overline{\beta}_g) = 1 \), not necessarily the semigroup of a plane branch. We use the notations and definitions from \fref{sec:semigroup}. Following \cite{teissier-appendix}, let \( (C^\Gamma, \boldsymbol{0}) \subset (X, \boldsymbol{0}) \) be the curve defined via the parameterization
\[ C^\Gamma : u_i = t^{\overline{\beta}_i}, \quad i \leq 0 \leq g, \]
where \( X := \mathbb{C}^{g+1} \). The germ \( (C^\Gamma, \boldsymbol{0} ) \) is irreducible since \( \gcd(\overline{\beta}_0, \dots, \overline{\beta}_g) = 1 \) and its local ring \( \mathcal{O}_{C^\Gamma, \boldsymbol{0}} \) equals
\[ \mathbb{C}\big\{C^\Gamma\big\} = \mathbb{C}\big\{ t^{\overline{\beta}_0}, \dots, t^{\overline{\beta}_g} \big\} \xhookrightarrow{\quad} \mathbb{C}\{t\}, \]
which has a natural structure of graded subalgebra of \( \mathbb{C}\{t\} \).
\begin{theorem}[{\cite[I.1]{teissier-appendix}}] \label{thm:teissier-appendix1}
Every branch \( (C, \boldsymbol{0}) \) with semigroup \( \Gamma \) is isomorphic to the generic fiber of a one parameter complex analytic deformation of \( (C^\Gamma, \boldsymbol{0}) \).
\end{theorem}
With extra structure on the semigroup \( \Gamma \) it is possible to obtain equations for \( (C^\Gamma, \boldsymbol{0}) \).
\begin{proposition}[{\cite[I.2]{teissier-appendix}}] \label{prop:teissier-appendix2}
If \( \Gamma \) satisfies \fref{lemma:semigroup}, the branch \( (C, \boldsymbol{0}) \subset (X, \boldsymbol{0}) \) is a quasi-homogeneous complete intersection with equations
\begin{equation} \label{eq:monomial-complete-intersection}
h_i := u_i^{n_i} - u_0^{l^{(i)}_0} u_1^{l^{(i)}_1} \cdots u_{i-1}^{l^{(i)}_{i-1}} = 0, \quad 1 \leq i \leq g,
\end{equation}
and weights \( \overline{\beta}_0, \overline{\beta}_1, \dots, \overline{\beta}_g \), where
\[ n_i \overline{\beta}_i = \overline{\beta}_{0} l_0^{(i)} + \cdots + \overline{\beta}_{i-1} l_{i-1}^{(i)} \in \langle \overline{\beta}_{0} , \dots, \overline{\beta}_{i-1} \rangle. \]
\end{proposition}

Applying the theory of miniversal deformations to the previous results Teissier proves the following result.

\begin{theorem}[{\cite[I.2]{teissier-appendix}}] \label{thm:teissier}
There exists a germ of a flat morphism
\[ p : (X_\Gamma, \boldsymbol{0}) \subset (X \times \mathbb{C}^{\tau_{-}}, \boldsymbol{0}) \longrightarrow (\mathbb{C}^{\tau_{-}}, \boldsymbol{0}) \]
consisting on the second projection from \( X_\Gamma \), such that it is a miniversal semigroup constant deformation of \( (C^\Gamma, \boldsymbol{0}) \) with the property that, for any branch \( (C, \boldsymbol{0}) \) with semigroup \( \Gamma \), there exists \( \boldsymbol{v}_C \in \mathbb{C}^{\tau_{-}} \) such that \( (p^{-1}(\boldsymbol{v}_C), \boldsymbol{0}) \) is analytically isomorphic to \( (C, \boldsymbol{0}) \).
\end{theorem}

The miniversal deformation in \fref{thm:teissier} can be made explicit, see \cite[I.2]{teissier-appendix}. Consider the Tjurina module of the complete intersection \( (C^\Gamma, \boldsymbol{0}) \),
\[ T^1_{C^\Gamma, \boldsymbol{0}} = \mathcal{O}_{X, \boldsymbol{0}}^g\Big/\left(J\boldsymbol{h} \cdot \mathcal{O}_{X, \boldsymbol{0}}^{g+1} + \langle h_1, \dots, h_g \rangle \cdot \mathcal{O}_{X, \boldsymbol{0}}^g\right), \]
where \( J \boldsymbol{h} \cdot \mathcal{O}_{X, \boldsymbol{0}}^{g+1} \) is the submodule of \( \mathcal{O}_{X, \boldsymbol{0}}^g \) generated by the columns of the Jacobian matrix of the morphism \( \boldsymbol{h} = (h_1, \dots, h_g) \). Since \( (C^\Gamma, \boldsymbol{0}) \) is an isolated singularity, \( T^1_{C^\Gamma, \boldsymbol{0}} \) is a finite dimensional \( \mathbb{C} \)--vector space of dimension \( \tau \). Moreover, since \( (C^\Gamma, \boldsymbol{0}) \) is Gorenstein, \(\tau = 2 \cdot \# (\mathbb{Z}_{>0} \setminus \Gamma)\), see \cite[Prop. 2.7]{teissier-appendix}.

\vskip 2mm

Let \( \boldsymbol{\phi}_1, \dots, \boldsymbol{\phi}_\mu \) be a basis of \( T^1_{X, \boldsymbol{0}} \). It is easy to see that we can take representatives for the vectors \( \boldsymbol{\phi}_r \) in \( \mathcal{O}_{X}^g \) having only one non-zero monomial entry \( \phi_{r, i} \). Since \( (C^\Gamma, \boldsymbol{0}) \) is quasi-homogeneous, we can endow \( T^1_{X, \boldsymbol{0}} \) with a structure of graded module in such a way that using only the elements \( \boldsymbol{\phi}_1, \dots, \boldsymbol{\phi}_{\tau_{-}} \) with negative weight, \( X_\Gamma \) is defined from \fref{eq:monomial-complete-intersection} by
\[ H_i := h_i + \sum_{r = 1}^{\tau_{-}} v_r \phi_{r, i}(u_0, \dots, u_g) = 0,  \quad 1 \leq i \leq g,\]
with the weight of \(\phi_{r,i} \) strictly bigger than \( n_i \overline{\beta}_i\), see \cite[Th. 2.10]{teissier-appendix}. One can check that the classes of the vectors \( (u_2, 0, \dots, 0), (0, u_3, 0, \dots, 0), \dots, (0, 0, \dots, u_g, 0) \) are \( \mathbb{C} \)--independent in \( T^1_{C^\Gamma, \boldsymbol{0}} \). Thus, if \( \Gamma \) is a plane branch semigroup, then \( n_i \overline{\beta}_i < \overline{\beta}_{i+1}, 1 \leq i < g \) and these vectors are part of the miniversal semigroup constant deformation of \( (C^\Gamma, \boldsymbol{0}) \).

\vskip 2mm

For \( \Gamma = \langle \overline{\beta}_0, \overline{\beta}_1, \dots, \overline{\beta}_g \rangle \) the semigroup of a plane branch, consider the following semigroup constant deformation of \( (C^\Gamma, \boldsymbol{0}) \)
\[ C: H_i' := h_i - u_{i+1} = 0, \quad 1 \leq i \leq g. \]
Define \( f_0 := x, f_1 := y \) and set recursively,
\begin{equation} \label{eq:monomial-curve-plane}
f_{i+1} := f_{i}^{n_i} - f_{0}^{l_0^{(i)}} f_1^{l_1^{(i)}} \cdots f_{i-1}^{l_{i-1}^{(i)}} \quad \text{with} \quad n_i \overline{\beta}_i = \overline{\beta}_{0} l_0^{(i)} + \cdots + \overline{\beta}_{i-1} l_{i-1}^{(i)},
\end{equation}
for \( 1 \leq i \leq g \). Then, we can embedded \( (C, \boldsymbol{0}) \) in the plane in such a way that it has equation \( f := f_{g+1} \) that looks like
\[ \left( \cdots \left( \left( \cdots \left( \left(y^{n_1} - x^{l_0^{(1)}}\right)^{n_2} \hspace{-0.2cm} - x^{l_0^{(2)}}y^{l_0^{(2)}}\right)^{n_3} \hspace{-0.3cm} - \cdots\right)^{n_i} \hspace{-0.3cm} - x^{l_0^{(i)}} \cdots f_{i-1}^{l_{i-1}^{(i)}}\right)^{n_{i+1}} \hspace{-0.45cm} - \cdots \right)^{n_{g-1}} \hspace{-0.45cm} -\, x^{l_0^{(g)}} \cdots f_{g-1}^{l^{(g)}_{g-1}}. \]
This proves the reverse implication in \fref{prop:characterization}.

\vskip 2mm

In the following sections, it will be essential to have a certain semigroup constant deformation of a plane branch equation, similar to \( f_{g+1} \) in \fref{eq:monomial-curve-plane}, in such a way that any other plane branch with the same semigroup is analytically equivalent to some fiber of that deformation. From the previous discussion we obtain the following result.
\begin{proposition} \label{prop:monomial-curve-plane}
Let \( \Gamma = \langle \overline{\beta}_0, \overline{\beta}_1, \dots, \overline{\beta}_g \rangle \) be a plane branch semigroup. Consider, with the same notation as above,
\begin{equation} \label{eq:plane-branch-def}
f_{i+1} = f_{i}^{n_i} - \lambda_i f_{0}^{l_0^{(i)}} f_1^{l_1^{(i)}} \cdots f_{i-1}^{l_{i-1}^{(i)}} + \hspace{-10mm} \sum_{\overline{\beta}_{0} k_0 + \cdots + \overline{\beta}_{i} k_{i} > n_{i} \overline{\beta}_{i}} \hspace{-10mm} t^{(i)}_{\underline{k}} f_0^{k_0}f_1^{k_1}\cdots f_{i}^{k_{i}},
\end{equation}
with \( \lambda_i \neq 0, \lambda_1 = 1 \) for \( i = 1, \dots, g \) and the sum being finite. Define, \[ f_{\boldsymbol{t}, \boldsymbol{\lambda}}(x, y) := f_{g+1}(x, y; \underline{t}^{(1)}, \dots, \underline{t}^{(g)}; \lambda_2, \dots, \lambda_g). \] Then, \( \{ f_{\boldsymbol{t}, \boldsymbol{\lambda}}(x, y)\}_{\boldsymbol{\lambda} \in \mathbb{C}^{g-1}} \) is an infinite family of semigroup constant deformations, all having semigroup \( \Gamma \), with the property that any other plane branch with semigroup \( \Gamma \) is analytically equivalent to a fiber of one element of the family.
\end{proposition}
\begin{proof}
Consider the semigroup constant deformation of the monomial curve \( (C^\Gamma, \boldsymbol{0}) \) from \fref{thm:teissier}. Since the semigroup \( \Gamma \) is of a plane branch, we can assume that the set of vectors \( (u_2, 0, \dots, 0), (0, u_3, 0, \dots, 0), \dots, (0, 0, \dots, u_g, 0) \) are part of the semigroup constant deformation. \( X_\Gamma \) will have equations
\[ C : H_i = h_i - v_{i+1} u_{i+1} + \sum_{r=g+1}^{\tau_-} v_r \phi_{r,i}(u_0, \dots, u_g) = 0,\quad 1 \leq i \leq g. \]
The embedding dimension of \( (C, \boldsymbol{0}) \) is equal to \( g + 1 - \text{rk}\, J\boldsymbol{H}( \boldsymbol{0} ) \), see \cite[4.3]{dejong-pfister00}. Since all the monomials in \( H_i \) have (non-weighted) degree bigger than 2, except for those in the vectors \( (u_2, 0, \dots, 0), (0, u_3, 0, \dots, 0), \dots, (0, 0, \dots, u_g, 0) \), the rank of the Jacobian is \( g - 1 \), if and only if, \( v_2 \cdots v_{g-1} \) is non-zero. Thus, the embedding dimension of \( (C, \boldsymbol{0} ) \) is 2, if and only if, all \(v_2, \dots, v_{g-1} \) are different from zero.

\vskip 2mm

Finally, performing elimination on the variables \( u_2, \dots, u_g \) one obtains a plane branch equation similar to \fref{eq:plane-branch-def} with \( \lambda_i = v_2^{n_2} \cdots v_i^{n_i} \neq 0 \) and a finite number of deformation monomials with coefficients that are polynomials in the variables \( v_r \). Therefore, there is an inclusion of the parameter space of \( X_\Gamma \) into the parameters of the family of deformations \( f_{\boldsymbol{t}, \boldsymbol{\lambda}} \).
\end{proof}

We will sometimes drop the dependency on the parameters \( \boldsymbol{\lambda} \in \mathbb{C}^{g-1} \), and denote just \( f_{\boldsymbol{t}}(x, y) \). Although we are considering a finite deformation of \( f_{\boldsymbol{t}} \), we can always assume that we have deformation terms of high enough order. Adding extra terms to the summation does not change the facts that the family contains all plane curves up to analytic isomorphism or that the deformation has constant semigroup.

\section{Resolution of plane curve singularities} \label{sec:resolution-plane-curves}

In this section, we will review some facts about resolution of singularities of plane curves. In the sequel, a plane curve will be a germ of a mapping \( f : (\mathbb{C}^2, \boldsymbol{0}) \longrightarrow (\mathbb{C}, 0) \) with \( f(\boldsymbol{0}) = 0 \). We will make a small abuse of notation and denote also by \( f \) an equation of the germ, assuming it is defined in a small neighborhood \( U \subseteq \mathbb{C}^2 \) of the origin. For irreducible plane curves \( f \), we will present the notions of toric resolution, maximal contact elements, and their relation. We will also see how some resolution data can be described in terms of the semigroup and the characteristic sequence of \( f \). In the last part of this section, we give equations for the resolution of the elements in the family \( f_{\boldsymbol{t}} \) considered in \fref{prop:monomial-curve-plane} at the so-called \emph{rupture divisors}. The majority of the results in this section are well-known and can be found in the books of Casas-Alvero \cite{casas} or Wall \cite{ctc-wall}.

\subsection{Resolution of singularities}

Let \(\pi : X' \longrightarrow U\) be \emph{the minimal embedded resolution} of the plane curve \( f \). Here, \(\pi\) is birational proper morphism and \( X' \) is a smooth surface. We can assume that \( \pi \) is given as a composition of point blow-ups
\[ \pi : X' := X_{s+1} \xrightarrow{\ \pi_{p_s} \ } X_r \longrightarrow \cdots \longrightarrow X_1 \xrightarrow{\ \pi_{p_0} \ } X_0 := U \subseteq \mathbb{C}^2, \]
with \( X_{i+1} := \textrm{Bl}_{p_i} X_i, p_i \in \textrm{Exc}(\pi_{p_{i-1}} \hspace{-4pt} \circ \cdots \circ \pi_{p_0}) \subset X_i\) for \( i = 1, \dots, s \) and \( p_0 := \boldsymbol{0} \). Denote by \( E := \textrm{Exc}({\pi}) \) the exceptional divisor of \( \pi \). Let \( K := \{p_0, p_1, \dots, p_s\} \) be the set of points that have been blown-up. Then, \( \{E_p\}_{p \in K} \) is the set of all the irreducible exceptional components of \( E \). The total transform divisor \( F_\pi \) and relative canonical divisor \( K_\pi \) will have the following expressions
\[ F_\pi := \textrm{Div}(\pi^* f) = \sum_{p \in K} N_p E_p + \mu_1 \widetilde{C}_1 + \cdots + \mu_t \widetilde{C}_t, \quad K_\pi = \textrm{Div}(\textrm{Jac}(\pi)) := \sum_{p \in K} k_p E_p, \]
where \( \widetilde{C}_1, \dots, \widetilde{C}_t \in \textrm{Div}_{\mathbb{Z}}(X') \) are the branches of the strict transform of \( f \), that is to say, \( \widetilde{C}_1 + \cdots + \widetilde{C}_t = \textrm{Div}(\overline{\pi^{-1}(f - \{\boldsymbol{0}\})}) \). The divisor \( F_\pi \) is a simple normal crossing divisor.

\vskip 2mm

We will distinguish between two types of exceptional divisors. An exceptional divisor is said to be of \emph{rupture} type if it intersects three or more divisors in the support of \( F_\pi \). It is said to be \emph{non-rupture} otherwise. Those exceptional divisors that only intersect one divisor in the support of \( F_\pi \) will be called \emph{dead-end} divisors. They are dead-end points in the dual graph of the resolution, hence the name.

\subsection{Toric resolutions and maximal contact elements} \label{sec:toric-resolution} We assume from now on that \( f : (\mathbb{C}^2, \boldsymbol{0}) \longrightarrow (\mathbb{C}, 0) \) is irreducible with semigroup \( \Gamma = \langle \overline{\beta}_0, \overline{\beta}_1, \dots, \overline{\beta}_g \rangle \) and we use the notations from \fref{sec:semigroup}. All plane branches having the same semigroup \( \Gamma \), or characteristic sequence, are \emph{equisingular}, see \cite[3.8]{casas}. This means that the combinatorics of the resolution is the same for all of them and can be determined by the characteristic sequence.

\vskip 2mm

A classical way to obtain the minimal resolution by point blow-ups of an irreducible plane curve \( f \) from its characteristic sequence \( (n, \beta_1, \dots, \beta_g) \) is using Enriques' theorem, see \cite[5.5]{casas}. We will however take the approach of Oka in \cite{oka96} and describe the minimal resolution of \( f \) as a composition of toric morphisms. Indeed, there exists a resolution map \( \pi \) of \( f \) that decomposes into \( g \geq 1 \) toric morphisms. For \( i = 1, \dots, g \),
\[ \pi^{(i)} := \pi_1 \circ \cdots \circ \pi_{i-1} \circ \pi_i : X^{(i)} \xrightarrow{\ \pi_i \ } X^{(i-1)} \xrightarrow{\pi_{i-1}} \cdots \xrightarrow{\ \pi_{2}\ } X^{(1)} \xrightarrow{\ \pi_1\ } X^{(0)} := U \subseteq \mathbb{C}^2, \]
where \( \pi_{i} \) is a toric morphism for a suitable choice of coordinates on \( X^{(i-1)} \) and \( \pi := \pi^{(g)} \). Each \( \pi_i \) resolves one characteristic exponent of the plane branch \( f \) in the sense that the strict transform of \( f \) on \( X^{(i)} \) has one characteristic exponent less than the strict transform on \( X^{(i-1)} \). In this way, \( X^{(i)} \) always contain one extra rupture divisor \( E_{p_i} \) than \( X^{(i-1)} \). We will denote by \( U_i, V_i \) the affine open sets, and by \( (x_i, y_i), (z_i, w_i) \) the coordinates, containing the \( i \)--th rupture divisor \( E_{p_i} \) on \( X^{(i)} \) after the \( i \)--th toric modification \( \pi_i \). In these coordinates, recalling the definitions of the integers \( n_i, q_i \) from \fref{sec:semigroup}, the toric morphism is given by
\begin{equation} \label{eq:equations-toric-resolution}
\pi_{i}(x_i, y_i) = \left(x_i^{n_i} y_i^{a_i}, x_i^{q_i} y_i^{b_i} \right) \quad \textrm{and} \quad \pi_i(z_i, w_i) = (z_i^{c_i} w_i^{n_i}, z_i^{d_i} w_i^{q_i}),
\end{equation}
with \( a_i, b_i, c_i, d_i \in \mathbb{Z}_{\geq 0} \) such that \( n_i b_i - q_i a_i = 1, q_i c_i - n_i d_i = 1 \) and \( a_i q_i + d_i n_i = n_i q_i - 1 \). These toric morphisms can be thought as a composition of point blow-ups. In the sequel, we will associate to each plane branch singularity these series of four integers \( a_i, b_i, c_i, d_i \) for every \( i = 1, \dots, g \). They are determined, although not explicitly, by the semigroup \( \Gamma \) of \( f \) since they depend on the continuous fraction expansion of \( q_i/n_i \).

\vskip 2mm

If the singularity is still not resolved at \( X^{(i)} \), one needs to perform an analytic change of coordinates around the unique singular point of the strict transform of \( f \) on \( X^{(i)} \) in order for \( \pi_{i+1} \) to be toric. These new coordinates, let us say \( (\bar{x}_i, \bar{y}_i) \) and \( (\bar{z}_i, \bar{w}_i) \), are such that \( \pi^{(i)}_* \bar{y}_i = \pi^{(i)}_* \bar{z}_i  \) is a germ \( f_i : (\mathbb{C}^2, \boldsymbol{0}) \longrightarrow (\mathbb{C}, 0) \) that is a \emph{maximal contact element} of \( f \) in the sense that the intersection of \( f_i \) and \( f \) is, precisely, \( \overline{\beta}_i \). By construction, each of these maximal contact elements \( f_i \) are resolved by the corresponding \( \pi^{(i)} \). In the case of the plane curves constructed in \fref{sec:monomial-curve}, the maximal contact element \( f_i \) coincide with the elements \( f_i \) defined in \fref{eq:monomial-curve-plane} and \fref{eq:plane-branch-def}. For completeness, we will assume that the \( (g+1) \)--th maximal contact element \( f_{g+1} \) is the curve \( f \) itself.

\vskip 2mm

It is easy to see that the semigroup \( \Gamma_{i+1} \) of the maximal contact element \( f_{i+1} \) is
\begin{equation} \label{eq:semigroup-max-contact}
\Gamma_{i+1} = \langle n_1 n_2 \cdots n_{i}, n_2 \cdots n_{i} \overline{m}_{1}, \dots, n_{i} \overline{m}_{i-1} ,\overbar{m}_{i} \rangle, \quad \textrm{for} \quad i = 1, \dots, g.
\end{equation}
Similarly, its characteristic sequence is given by \((n_1 n_2 \cdots n_{i}, n_2 \cdots n_{i} m_1, \dots, n_{i} m_{i-1}, m_{i})\). Apart from this numerical data of the maximal contact elements, we are interested in describing the multiplicities of the total transform and the relative canonical divisor along the rupture and dead-end divisors in terms of the semigroup \( \Gamma \). Following the same notation as above, denote \( E_{p_i}, i = 1, \dots, g \) the rupture divisors of \( f \). Similarly, denote \( E_{q_i} \) for \( i = 0, \dots, g \) the dead-end divisors. It is well-known, see \cite[8.5]{ctc-wall}, that
\begin{equation} \label{eq:numeric-data-semigroup}
N_{p_i} = n_i \overline{\beta}_i, \quad k_{p_i} + 1 = m_i + n_1 \cdots n_i, \quad N_{q_i} = \overline{\beta}_i, \quad k_{q_i} + 1 = \lceil (m_i + n_1 \cdots n_i)/n_i \rceil .
\end{equation}

We will end this section with a technical result about the resolution of the elements of the constant semigroup deformations \( \{ f_{\boldsymbol{t}, \boldsymbol{\lambda}}\}_{\boldsymbol{\lambda} \in \mathbb{C}^{g-1}} \). Having constant semigroup means that all the fibers of all the elements of the family are equisingular. Hence, the toric resolution of the plane branches in the family \( f_{\boldsymbol{t}} \) is the same modulo the coordinates needed at each \( X^{(i)} \). The following proposition describes locally the equations of \( f_{\boldsymbol{t}} \) around the rupture divisors after pulling back by \( \pi^{(i)} \).

\begin{proposition} \label{prop:equations-curve}
Let \( E_{p_i} \) be the \( i \)--th rupture divisor on the surface \( X^{(i)} \) and let \( U_i, V_i \) be the corresponding charts containing \( E_{p_i} \) with local coordinates \( (x_i, y_i) \) and \( (z_i, w_i) \), respectively. Then,
\begin{itemize}[leftmargin=*]
\item The equations of the total transform of \( f_{\boldsymbol{t}}(x, y) \) are given by
\begin{equation} \label{eq:prop-equations-curve-1}
x_i^{n_i \overline{\beta}_i} y_i^{a_i \overline{\beta}_i} u_1(x_i, y_i)\tilde{f}_{\boldsymbol{t}}(x_i, y_i), \qquad z_i^{(c_i n_{i-1} \overline{m}_{i-1} + d_i)e_{i-1}} w_i^{n_i \overline{\beta}_i} u_2(x_i, y_i) \tilde{f}_{\boldsymbol{t}}(z_i, w_i),
\end{equation}
where \( u_1, u_2 \) are units at any point of \( E_{p_i} \).
\vskip 2mm
\item The equations \( \tilde{f}_{\boldsymbol{t}} \) of the strict transform of \( f \) are
\[
\tilde{f}_{\boldsymbol{t}}(x_i, y_i) = \tilde{f}_{g+1}(x_i, y_i; \underline{t}^{(1)}, \dots, \underline{t}^{(g)}), \quad \tilde{f}_{\boldsymbol{t}}(z_i, w_i) = \tilde{f}_{g+1}(z_i, w_i; \underline{t}^{(1)}, \dots, \underline{t}^{(g)}).
\]
\item The \( (i+1) \)--th maximal contact element has equations
\begin{equation} \label{eq:prop-equations-curve-2}
\begin{split}
\tilde{f}_{i+1}(x_i, y_i) & = y_i - \lambda_i + \hspace{-10mm} \sum_{\overline{\beta}_{0} k_0 + \cdots + \overline{\beta}_{i} k_{i} > n_{i} \overline{\beta}_{i} } \hspace{-10mm} t_{\underline{k}}^{(i)} x_i^{\rho_{i+1}^{(i)}(\underline{k})} y_i^{A^{(i)}_{i+1}(\underline{k})} u^{(i)}_{\underline{k}}(x_i, y_i), \qquad (U_i \ \textrm{chart}\,), \\
\tilde{f}_{i+1}(z_i, w_i) & = 1 - \lambda_i z_i + \hspace{-10mm} \sum_{\overline{\beta}_{0} k_0 + \cdots + \overline{\beta}_{i} k_{i} > n_{i} \overline{\beta}_{i} } \hspace{-10mm} t_{\underline{k}}^{(i)} z_i^{C^{(i)}_{i+1}(\underline{k})} w_i^{\rho_{i+1}^{(i)}(\underline{k})} u^{(i)}_{\underline{k}}(z_i, w_i), \qquad (V_i \ \textrm{chart}\,),
\end{split}
\end{equation}
where \( u_{\underline{k}}^{(i)} \) are units at any point of \( E_{p_i} \).
\vskip 2mm
\item The remaining maximal contact elements \( f_{j+1}, j > i \) have strict transforms given by
\begin{equation} \label{eq:prop-equations-curve-3}
\tilde{f}_{j+1} = \tilde{f}_j^{n_j} - \lambda_j x_i^{\rho_{j+1}^{(i)}(\underline{l}_j)} y_i^{A_{j+1}^{(i)}(\underline{l}_j)} u_{\underline{0}}^{(j)} \tilde{f}_{i+1}^{l_{i+1}^{(j)}} \cdots \tilde{f}_{j-1}^{l_{j-1}^{(j)}} + \hspace{-11mm} \sum_{\overline{\beta}_{0} k_0 + \cdots + \overline{\beta}_{j} k_{j} > n_{j} \overline{\beta}_{j} } \hspace{-10mm} t_{\underline{k}}^{(j)} x_i^{\rho_{j+1}^{(i)}(\underline{k})} y_i^{A_{j+1}^{(i)}(\underline{k})} u^{(j)}_{\underline{k}} \tilde{f}_{i+1}^{k_{i+1}} \cdots \tilde{f}_j^{k_j}
\end{equation}
in the \( U_i \) chart, and similarly in \( V_i \). As before, \( u_{\underline{0}}^{(j)}, u_{\underline{k}}^{(j)} \) are units everywhere on \( E_{p_i} \) and we denote \( \underline{l}_j := \big(l_0^{(j)}, l_1^{(j)}, \dots, l_{j-1}^{(j)}, 0\big) \) the integers from \fref{prop:teissier-appendix2}.
\end{itemize}

\vskip 2mm

\item Finally, \( \rho_{j+1}^{(i)}, A_{j+1}^{(i)}, C_{j+1}^{(i)}, j \geq i \), are the following linear forms:
\begin{equation} \label{eq:prop-equations-curve-4}
\begin{split}
\rho_{j+1}^{(i)}(\underline{k}) & = \sum_{l = 0}^i n_{l+1} \cdots n_i \overline{m}_l k_l  + n_i \overline{m}_i \sum_{l=i+1}^j n_{i+1} \cdots n_{l-1} k_l - n_i \cdots n_j \overline{m}_i, \\
A_{j+1}^{(i)}(\underline{k}) & = a_i \sum_{l = 0}^{i-1} n_{l+1} \cdots n_{i-1} \overline{m}_l k_l + (a_i n_{i-1} \overline{m}_{i-1} + b_i)k_i \\
  & + a_i \overline{m}_i \sum_{l = i+1}^j n_{i+1} \cdots n_{l-1} k_l - a_i \overline{m}_i n_{i+1} \cdots n_j,\\
C_{j+1}^{(i)}(\underline{k}) & = c_i \sum_{l = 0}^{i-1} n_{l+1} \cdots n_{i-1} \overline{m}_l k_l + (c_i n_{i-1} \overline{m}_{i-1} + d_i)k_i \\
  & + n_i(c_i n_{i-1} \overline{m}_{i-1} + d_i) \sum_{l = i+1}^j n_{i+1} \cdots n_{l-1} k_l - (c_i n_{i-1} \overline{m}_{i-1} + d_i) n_i \cdots n_j.
\end{split}
\end{equation}
\end{proposition}
\begin{proof}
The results follow from the inductive procedure of applying the toric transformations from \fref{eq:equations-toric-resolution} to the equations of \( f_{\boldsymbol{t}} \) in \fref{prop:monomial-curve-plane}. At each \( X^{(i)} \) the analytic coordinates which make the morphism \( \pi_i \) toric are described by \( \bar{y}_i = \widetilde{f}_{i+1}, \bar{x}_i = x_i u_i \) in the \( U_i \) chart, for some unit \( u_i \). The expressions for the linear forms \( \rho_{j+1}^{(i)}, A_{j+1}^{(i)}, C_{j+1}^{(i)} \) follow, recursively, from the relations
\[
\begin{split}
\rho_{j+1}^{(i)}(\underline{k}) & = n_i \rho_{j+1}^{(i-1)}(\underline{k}) + q_i k_i + n_i q_i \sum_{l = i+1}^j n_{i+1} \cdots n_{l-1} k_l - n_i \cdots n_j q_i, \\
A_{j+1}^{(i)}(\underline{k}) & = a_i \rho_{j+1}^{(i-1)}(\underline{k}) + b_i k_i + a_i q_i \sum_{l = i+1}^j n_{i+1} \cdots n_{l-1} k_l - a_i q_i n_{i+1} \cdots n_j, \\
C_{j+1}^{(i)}(\underline{k}) & = c_i \rho_{j+1}^{(i-1)}(\underline{k}) + d_i k_i + n_i d_i \sum_{l = i+1}^j n_{i+1} \cdots n_{l-1} k_l - d_i n_i \cdots n_j.
\end{split}
\]\end{proof}
Following the notations from \fref{prop:equations-curve}, we will fix, in the sequel, the index \( i \) to denote that we have resolved the singularity up to the \( i \)--th rupture divisor on \( X^{(i)} \). On the other hand, the index \( j \) will make reference to the \( j \)--th maximal contact element. At any step of the resolution process, one has that \( 1 \leq i \leq j \leq g \).

\begin{lemma} \label{lemma:lemma-ac1}
Let \( \rho_{j+1}^{(i)}(\underline{k}), A_{j+1}^{(i)}(\underline{k}), C_{j+1}^{(i)}(\underline{k}) \) be the linear forms in \fref{prop:equations-curve}. Then,
\[ A_{j+1}^{(i)}(\underline{k}) + C_{j+1}^{(i)}(\underline{k}) + \sum_{l = i+1}^j n_{i+1} \cdots n_{l-1} k_l = \rho_{j+1}^{(i)}(\underline{k}) + n_{i+1} \cdots n_j. \]
\end{lemma}
\begin{proof}
From the relations between \( a_i, b_i, c_i, d_i \), one can deduce that \( a_i + c_i = n_i, b_i + d_i = q_i \). The result follows from adding \( A_{j+1}^{(i)} \) and \( C_{j+1}^{(i)} \) and using these relations.
\end{proof}

\section{Poles and residues for plane curves} \label{sec:poles-plane-curves}

Let \( f : (\mathbb{C}^{2}, \boldsymbol{0}) \longrightarrow (\mathbb{C}, 0) \) be a plane curve not necessarily reduced or irreducible. After fixing local coordinates \( x, y \), and with a small abuse of notation, let \( f \in \mathbb{C}\{x, y\}\) be an equation of the germ, assuming it is defined in a neighborhood \( U \subseteq \mathbb{C}^2 \) of the origin. Following \fref{sec:regularization}, define the complex zeta function \( f^s \) of a local singularity as
\begin{equation} \label{eq:definition-local-sing}
\langle f^s, \varphi \rangle := \int_{U} |f(x, y)|^{2s} \varphi(z)\, dz, \quad \textrm{for} \quad \textrm{Re}(s) > 0,
\end{equation}
with \( \varphi \in C^\infty_c(U) \) and \( z := (x, y, \bar{x}, \bar{y}) \). The poles of \( \langle f^s, \varphi \rangle \) do not depend on the equation of the germ or the local coordinates \( x, y \). As discussed in \fref{sec:regularization}, \( f^s \) must be understood in the distributional sense. In this section, we will use the minimal resolution of singularities of \( f \) to study the structure of the residues of \( f^s \) at any candidate pole \( \sigma \). The residue will be expressed as an improper integral along the exceptional divisor associated to \( \sigma \). In order to do so, we first present the straightforward generalization of \fref{prop:regularization} to the two dimensional case and see how poles of order two might arise. Finally, we will use the residue formula to prove that most non-rupture divisors do not contribute to poles of the complex zeta function \( f^s \).

\subsection{Regularization of monomials in two variables} \label{sec:poles-order-two}

The result from \fref{prop:regularization} can be easily generalized to the two dimensional case, mimicking the proof in \cite{gelfand-shilov}, to see how poles of order two arise. Let \( \varphi(z_1, z_2) \in C_c^\infty(\mathbb{C}^2) \) which is, in fact, a function of \( z = (z_1, z_2, \bar{z}_1, \bar{z}_2) \), and consider
\begin{equation} \label{eq:regularization-2}
\langle z_1^{s_1} z_2^{s_2}, \varphi \rangle = \int_{\mathbb{C}^2} |z_1|^{2s_1} |z_2|^{2s_2} \varphi(z) dz,
\end{equation}
which is absolutely convergent for \( \textrm{Re}(s_1) > -1 \), \( \textrm{Re}(s_2) > -1 \) since \( \varphi \) has compact support.

\vskip 2mm

Let \( \Delta_0 = D_1 \times D_2  \) be the polydisc formed by the discs of radius one centered at the origin, i.e. \( D_1 = \{ |z_1| \leq 1 \} \) and \( D_2 = \{ |z_2| \leq 1\} \). We can decompose \( \mathbb{C}^2 \) as the disjoint union
\[ \Delta_0 \cup (D_1 \times \mathbb{C} \setminus D_2) \cup (\mathbb{C} \setminus D_1 \times D_2) \cup (\mathbb{C} \setminus D_1 \times \mathbb{C} \setminus D_2). \]
Using the notation \( z^k = z_1^{k_1} z_2^{k_2} \bar{z}_1^{k_3} \bar{z}_2^{k_4} \), the integral in \fref{eq:regularization-2} on the region \( \Delta_0 \) can be written as
\[ \int_{\Delta_0} |z_1|^{2 s_1} |z_2|^{2 s_2} \bigg(\varphi(z) - \sum_{|k| \leq m} \frac{\partial^{k} \varphi(\boldsymbol{0})}{\partial z^k} \frac{z^k}{k!} \bigg) dz - \sum_{|k| \leq m}
\frac{\partial^{k} \varphi}{\partial z^{k}}(\boldsymbol{0})\frac{4 \pi^2}{k!(s_1 + k_1 + 1)(s_2 + k_2 + 1)},\]
where in the second summation, we have \(k_1 = k_3, k_2 = k_4\). The left-hand integral is holomorphic on the regions \( \textrm{Re}(s_1) > -m-1, \textrm{Re}(s_2) > -m -1 \).

\vskip 2mm

With a small abuse of the notation, let \( z_1 = (z_1, \bar{z}_1), z_2 =(z_2, \bar{z}_2) \) and \( z_1^{k_1} = z_1^{k_{1,1}} \bar{z}_1^{k_{1,2}} \). On the region \( D_1 \times \mathbb{C} \setminus D_2 \), the integral in \fref{eq:regularization-2} is
\[ \int_{D_1 \times \mathbb{C} \setminus D_2} \hspace{-20pt} |z_1|^{2s_1} |z_2|^{2s_2} \bigg( \varphi(z) - \sum_{|k_1| \leq m} \frac{\partial^{k_1} \varphi}{\partial z_1^{k_1}}(0, z_2) \frac{z_1^{k_1}}{k_1!} \bigg) dz_1 dz_2 - 2 \pi i \sum_{|k_1| \leq m} \displaystyle \frac{\displaystyle \int_{|z_2|>1} \hspace{-2pt} |z_2|^{2s_2} \frac{\partial^{k_1} \varphi}{\partial z_1^{k_1}} (0, z_2) dz_2}{(k_{1,1}!)^2(s_1 + k_{1,1} + 1)}, \]
where in the second sum, \( k_{1,1} = k_{1,2} \). The left-hand integral is holomorphic in \( \textrm{Re}(s_1) > - m - 1 \). By symmetry, a similar expression holds true in the other region, \( \mathbb{C} \setminus D_1 \times D_2 \). On the last region, \( \mathbb{C} \setminus D_1 \times \mathbb{C}\setminus D_2 \), the integral in \fref{eq:regularization-2} is absolutely convergent for all \( s_1, s_2 \in \mathbb{C} \).

\vskip 2mm

From the regularization of \( z_1^{s_1} z_2^{s_2} \) constructed above we can see that the residue of \( z_1^{s_1} z_2^{s_2} \) at a simple pole \( s_1 = -k-1, k \in \mathbb{Z}_{>0} \), i.e. the coefficient of \( (s + k + 1)^{-1} \), is given by the following function of \( s_2 \)
\begin{equation} \label{eq:residue-2}
\Res_{s_1 = -k-1} \langle z_1^{s_1} z_2^{s_2}, \varphi \rangle = -\frac{2 \pi i}{(k!)^2} \int_{\mathbb{C}} |z_2|^{2s_2} \frac{\partial^{2k} \varphi}{\partial z_1^k \partial \bar{z}_1^k}(0, s_2)\, dz_2 d\bar{z}_2.
\end{equation}
The residue in \fref{eq:residue-2} being a function of \( s_2 \) implies that \( \Res_{s_1 = -k-1} z_1^{s_1}z_2^{s_2} \)  will have a simple pole, as a function of \( s_2 \), on the poles of order two of \( z_1^{s_1} z_2^{s_2} \). Conversely, if \fref{eq:residue-2} is zero for certain \( s_2 = \alpha \), the point \( (s_1, s_2) = (-k -1, \alpha) \) is neither a pole of order one nor a pole of order two of \( z_1^{s_1} z_2^{s_2} \).

\subsection{The residue at the poles} \label{sec:residues-poles}

Let \( E_{p}, p \in K \) be an irreducible exceptional divisor. We will denote by \( D_{1}, D_{2}, \dots, D_{r} \in \textrm{Div}_{\mathbb{Z}}(X') \) the other prime components (exceptional or not) of \( F_\pi \) crossing \( E_p \). By definition, dead-end divisors have only one divisor crossing them, which will be denoted \( D_1 \). On the other hand, rupture divisors have at least three divisors crossing them, i.e. \( r \geq 3 \). In any other case, \( r = 2 \). We will denote by \( N_1, N_2, \dots, N_r \) (resp. \( k_1, k_2, \dots, k_r \)) the coefficients of \( D_1, D_2, \dots, D_r \) in \( F_\pi \) (resp. in \( K_\pi \)). Since no confusion arises, we drop the explicit dependence on \( p \in K \).

\vskip 2mm

For each \( E_p, p \in K \) we consider two affine charts \( U_p, V_p \) containing \( E_p \) that arise after the blow-up of a neighborhood of \( p \) in any chart containing \( p \). The origin of the charts \( U_p, V_p \) are neighborhoods of opposite points in the projective line \( E_p \). Usually, these points are the intersection points of \( E_p \) with two other components of \( F_\pi \) which we will assume to be \( D_1 \) and \( D_2 \). For simplicity, if this is not the case, we will set \( N_1, k_1 \) (or \( N_2, k_2 \)) to be zero. The only case in which both \( N_1D_1 \) and \( N_2D_2 \) are zero is when the minimal resolution \( \pi \) consists of a single blow-up, that is, an homogeneous singularity.

\vskip 2mm

In order to define the complex zeta function \( f^s \) on \( X' \), we need to work locally with coordinates. Accordingly, let \( (x_p, y_p), (z_p, w_p) \) be the natural holomorphic coordinates of \( U_p, V_p \) centered at the origins of both charts which, by construction, are the origin and the infinity point on a \( \mathbb{P}^1_\mathbb{C} \), or vice versa. The coordinates \( (x_p, y_p), (z_p, w_p) \) are related at the intersection \( U_p \cap V_p \) by \(x_p = z_p^{\kappa_p}w_p, y_p z_p = 1, \) and where the integer value \( \kappa_p \in \mathbb{Z}_{> 0} \) has a very precise geometric meaning, namely, \( \kappa_p = - E_p \cdot E_p \).

\vskip 2mm

Following the discussion in \fref{sec:regularization}, each exceptional component \( E_p \) contributes with a sequence of candidate poles to the meromorphic continuation of \( f^s \). Indeed, with the notations above,
\[ \left\{ \sigma_{p, \nu} = -\frac{k_p + 1 + \nu}{N_p}\ \bigg|\ \nu \in \mathbb{Z}_{\geq 0} \right\}, \quad p \in K. \]
The set \( \mathcal{W}_K := \{U_p, V_p\}_{p \in K} \) forms a finite affine open cover \( X' \), which applied to the construction presented in \fref{sec:regularization}, \fref{eq:anal-cont0}, results in the following proposition. First, denote by \( \{\eta_{1, q}, \eta_{2, q}\}_{q \in K} \) a partition of unity subordinated to the open cover \( \mathcal{W}_K \).

\begin{proposition} \label{prop:integral-along-divisor}
Using the affine open cover \( \mathcal{W}_K \), the part of the complex zeta function \( f^s \) on \fref{eq:anal-cont0} involving the divisor \( E_p \) can be written as the sum of the integrals over just two affine charts \( U_p, V_p \in \mathcal{W}_K \) containing \( E_p \), namely
\begin{equation} \label{eq:integral-divisor}
\begin{split}
  & \int_{U_p} |x_p|^{2(N_p s + k_p)} |y_p|^{2(N_1s + k_1)} \Phi_1(x_p, y_p; s) \eta_1\, dx_p dy_p d\bar{x}_p d\bar{y}_p \quad + \\
  & \int_{V_p} |z_p|^{2(N_2s + k_2)} |w_p|^{{2(N_p s + k_p)}} \Phi_2(z_p, w_p; s) \eta_2\, dz_p dw_p d\bar{z}_p d\bar{w}_p,
\end{split}
\end{equation}
where \( \Phi_1(x_i, y_i; s), \Phi_2(z_i, w_i; s) \) are infinitely many times differentiable at neighborhoods of the points \( p_1 = E_{p} \cap D_1 \) and \( p_2 = E_{p} \cap D_2 \). More precisely,
\[ \Phi_1 := |u_1|^{2s} |v_1|^2 (\pi^* \varphi)\restr{U_p}, \quad \Phi_2 := |u_2|^{2s} |v_2|^2 (\pi^* \varphi)\restr{V_p}, \]
with the elements \( u_1, v_1 \) (resp. \( u_2, v_2 \)) being units in the local ring at the points \( p_1 \) (resp. \( p_2 \)). Finally, \(\eta_1\) and \( \eta_2 \) have compact support and \( \eta_1\restr{E_p} +\, \eta_2\restr{E_p} \equiv 1 \).
\end{proposition}
\begin{proof}
Let us denote by \( \mathcal{W}_{p_1}, \mathcal{W}_{p_2} \) all the elements in \( \mathcal{W}_K \) that contain \( p_1 \) and \( p_2 \). By construction, \( \mathcal{W}_{p_1} \) is disjoint from \( \mathcal{W}_{p_2} \), since there can be no affine open set containing both \( p_1 \) and \( p_2 \), and the union of \( \mathcal{W}_{p_1} \) and \( \mathcal{W}_{p_2} \) contains all the charts from \( \mathcal{W}_K \) containing \( E_p \). Applying \fref{eq:anal-cont0} from \fref{sec:regularization}, the part of \( f^s \) on \( X' \) where the divisor \( E_p \) appears is a sum of the integrals over the affine open sets from \( \mathcal{W}_{p_1} \) and \( \mathcal{W}_{p_2} \),
\begin{equation} \label{eq:proof-zeta-function-piece}
\sum_{U_{q} \in \mathcal{W}_{p_1}} \hspace{-2pt} \int_{U_{q}} \hspace{-2pt} |\pi^* f|^{2s}\restr{U_q} (\pi^* \varphi)\restr{U_q} \left|d \pi\restr{U_q}\hspace{-1pt}\right|^2 \eta_{1, q} + \hspace{-2pt} \sum_{V_{q} \in \mathcal{W}_{p_2}} \hspace{-2pt} \int_{V_{q}} \hspace{-2pt} |\pi^* f|^{2s}\restr{V_q} (\pi^* \varphi)\restr{V_q} \left|d \pi\restr{V_q}\hspace{-1pt}\right|^2 \eta_{2, q}.
\end{equation}

\vskip 2mm

Let us see that we can reduce \fref{eq:proof-zeta-function-piece} to \fref{eq:integral-divisor}. The proof is the same for both summations in \fref{eq:proof-zeta-function-piece}. Since the elements of \( \mathcal{W}_{p_1} \) are blow-up charts, the difference between the union and the intersection of all the elements in \( \mathcal{W}_{p_1} \) is a finite number of lines. Given that a finite number of lines have measure zero, they do not affect the integral, and we can replace the left-hand summation of \fref{eq:proof-zeta-function-piece} by
\begin{equation*}
\int_{\cap U_q} |\pi^* f|^{2s}\restr{\cap U_q} (\pi^* \varphi)\restr{\cap U_q} |d \pi\restr{\cap U_q}\hspace{-2pt}|^2 \eta_{1}
\end{equation*}
This equality is true since \( |\pi^* f|^{2s} \pi^*\varphi |d \pi|^2 \) is a global section on \( X' \) and it coincides at the intersection of all the \( U_q \in \mathcal{W}_{p_1} \). Concerning the partitions of unit, we just set \( \eta_1 := \sum_{U_q \in \mathcal{W}_{p_1}} \eta_{1, q} \). Finally, by the same argument as before, we can replace \( \cap_{U_q \in \mathcal{W}_{p_1}} U_q \) by any \( U_p \in \mathcal{W}_{p_1} \) yielding \fref{eq:integral-divisor}. Notice that, by definition of \( \mathcal{W}_{p_1} \), no other \(\eta \in \{\eta_{1, q}, \eta_{2, q} \}_{q \in K} \), except for those in \( \eta_1 \), has \( p_1 \) in its support.
\end{proof}

\vskip 2mm

Before presenting the formula for the residue, let us introduce the following rational numbers associated to a candidate pole \( \sigma_{p, \nu} \) of an irreducible exceptional divisor \( E_p, p \in K \). They will play an important role in the analysis of the residues.
\begin{definition}[Residue numbers]
Let \( \sigma_{p, \nu}, \nu \in \mathbb{Z}_{\geq 0} \) be a candidate pole of \( f^s \) associated to an exceptional divisor \( E_p, p \in K \) intersecting the divisors \( D_1, D_2, \dots, D_r \in \textrm{Div}_{\mathbb{Z}}(X') \). Define the \emph{residue numbers} as
\begin{equation} \label{eq:definition-epsilons}
\epsilon_{i, \nu} := N_i \sigma_{p, \nu} + k_i \in \mathbb{Q} \quad \textrm{for} \quad i = 1, \dots, r.
\end{equation}
\end{definition}
For the ease of notation, we will omit the dependence of \( \epsilon_{1, \nu}, \epsilon_{2, \nu} \) on \( p \in K \). The following relations between \( \epsilon_{1, \nu}, \epsilon_{2, \nu}, \dots, \epsilon_{r, \nu} \) holds.
\begin{lemma} \label{lemma:relation-epsilon}
For any \( \nu \in \mathbb{Z}_{\geq 0} \), we have
\begin{equation} \label{eq:relation-epsilon}
\epsilon_{1, \nu} + \epsilon_{2, \nu} + \cdots + \epsilon_{r, \nu} + \kappa_p \nu + 2 = 0.
\end{equation}
\end{lemma}

\begin{proof}
Consider the \( \mathbb{Q} \)--divisor \( \sigma_{p, \nu} F_\pi + K_\pi \). Applying the adjunction formula for surfaces \cite[\S V.1]{hartshorne}, \( (\sigma_{p, \nu} F_\pi + K_\pi) \cdot E_p = \kappa_p - 2 \), recall \( \kappa_p = - E_p \cdot E_p \). On the other hand, \[ (\sigma_{p, \nu} F_\pi + K_\pi) \cdot E_p = \sum_{i = 0}^r \epsilon_{i, \nu} - \kappa_p (N_p \sigma_{p, v} + k_p). \] Since \( N_p \sigma_{p, \nu} + k_p = - \nu - 1 \), the result follows.
\end{proof}

A first instance of the numbers \( \epsilon_{i, \nu} \) and of \fref{eq:relation-epsilon}, in the case of rupture divisors of irreducible plane curves and \( \nu = 0 \), already appeared in an article of Lichtin \cite{lichtin85}.

\vskip 2mm

The formula for the residue at a candidate pole \( \sigma_{p, \nu} \) is presented next. The residue is expressed as an improper integral, see \fref{rmk:remark-integral} below, along the divisor \( E_p \) having singularities of orders \( \epsilon_{1, \nu}, \epsilon_{2, \nu}, \dots, \epsilon_{1, r} \) at the intersection points of \( E_p \) with \( D_1, D_2, \dots, D_r \).

\begin{proposition} \label{prop:residue}
The residue of the complex zeta function \( f^s \) at a candidate pole \( s = \sigma_{p, \nu} \) is given by
\begin{equation}
\begin{split}
\Res_{\ s = \sigma_{p, \nu}} \langle f^s, \varphi \rangle & = \frac{-2 \pi i}{(\nu!)^2} \int_{\mathbb{C}} |y_p|^{2\epsilon_{1, \nu}} \frac{\partial^{2\nu} \Phi_1}{\partial x_p^{\nu} \partial \bar{x}_p^\nu}(0, y_p; \sigma_{p, \nu})\, dy_p d\bar{y}_p \qquad \,\,\,\,(U_p \ \textrm{chart}\,) \\
                                              & = \frac{-2 \pi i}{(\nu!)^2} \int_{\mathbb{C}} |z_p|^{2\epsilon_{2, \nu}} \frac{\partial^{2\nu} \Phi_2}{\partial w_p^\nu \partial \bar{w}_p^\nu}(z_p, 0; \sigma_{p, \nu})\, dz_p d\bar{z}_p, \qquad (V_p \ \textrm{chart}\,).
\end{split}
\end{equation}
\end{proposition}
\begin{proof}
Applying \fref{eq:residue-2} to \fref{prop:integral-along-divisor} with \( s_2 = N_1 \sigma_{p, \nu} + k_1 \) and \( N_2 \sigma_{p, \nu} + k_2\), respectively, we obtain that the residue of \( f^s \) at \( s = \sigma_{p, \nu} \) is
\begin{equation} \label{eq:prop-residue-1}
\begin{split}
\Res_{\ s = \sigma_{p, \nu}} \langle f^s, \varphi \rangle = \frac{-2 \pi i}{(\nu!)^2} \bigg( & \int_{\mathbb{C}} |y_p|^{2(N_1\sigma_{p, \nu} + k_1)} \frac{\partial^{2\nu} \Phi_1 \eta_1}{\partial x_p^\nu \partial \bar{x}_p^\nu}(0, y_p; \sigma_{p, \nu})\, dy_pd\bar{y}_p \\
                            + & \int_{\mathbb{C}} |z_p|^{2(N_2\sigma_{p, \nu} + k_2)} \frac{\partial^{2\nu} \Phi_2 \eta_2}{\partial w_p^\nu \partial \bar{w}_p^\nu}(z_p, 0; \sigma_{p, \nu})\, dz_p d\bar{z}_p \bigg).
\end{split}
\end{equation}
The global section \( |\pi^* f|^{2s} (\pi^* \varphi) |d \pi|^2 \) restricted to \( U_p, V_p \) contains \( \Phi_1, \Phi_2 \). The whole global section differs only from \( \Phi_1, \Phi_2 \) by the exceptional part in the total transform \( | \pi^* f |^{2s} \). Thus, at the intersection \( U_p \cap V_p \), having that \( y_p z_p = 1, w_p = x_p^{\kappa_p} y_p \), one checks that
\[
\begin{split}
\Phi_2(z_p, w_p; \sigma_{p, \nu}) & = \Phi_2(y_p^{-1}, x_p^{\kappa_p}y_p; \sigma_{p, \nu}) \\
& = |y_p|^{-2 (\epsilon_{3, \nu} + \cdots + \epsilon_{r, \nu})} \Phi_2(y_p, x_p^{\kappa_p} y_p; \sigma_{p, \nu}) = |y_p|^{-2 (\epsilon_{3, \nu} + \cdots + \epsilon_{r, \nu})} \Phi_1(x_p, y_p; \sigma_{p, \nu}).
\end{split}
\]
Now, deriving both sides with respect to \( w_p \) and \( \overbar{w}_p \) and setting \( w_p, \bar{w}_p = 0 \) yields,
\[
 \frac{\partial^{2\nu} \Phi_2}{\partial w_p^\nu \partial \overbar{w}_p^\nu}(z_p, 0; \sigma_{p, \nu}) = |y_p|^{-2(\epsilon_{3, \nu} + \cdots + \epsilon_{r, \nu} + \kappa_p \nu)} \frac{\partial^{2\nu} \Phi_1}{\partial x_p^\nu \partial \bar{x}_p^\nu}(0, y_p; \sigma_{p, \nu}).
\]
This, together with \fref{lemma:relation-epsilon}, shows that the differential forms
\[
  |y_p|^{2 \epsilon_{1,\nu}} \frac{\partial^{2\nu} \Phi_1}{\partial^\nu x_p \partial^\nu \bar{x}_p}(0, y_p; \sigma_{p, \nu}) \dd y_p \wedge \dd \bar{y}_p, \quad |z_p|^{2 \epsilon_{2, \nu}} \frac{\partial^{2\nu} \Phi_2}{\partial w_p^\nu \partial \bar{w}_p}(z_p, 0; \sigma_{p, \nu}) \dd w_p \wedge \dd \bar{w}_p,
\]
define a global section on \( E_p \). As a consequence, it suffices to use \( z_p y_p = 1 \) in either of the integrals in \fref{eq:prop-residue-1}, together with the fact that \( \eta_1\restr{E_p} +\, \eta_2\restr{E_p} \equiv 1 \).
\end{proof}

\begin{corollary} \label{cor:corollary-residue}
The residue of the complex zeta function \( f^s \) at \( s = \sigma_{p, \nu} \) is given, in the \( U_p \) chart, by
\begin{equation} \label{eq:corollary-residue}
\begin{split}
\Res_{s = \sigma_{p, \nu}} \langle f^s, \varphi \rangle = &\ \frac{-2 \pi i}{(\nu!)^2} \int_{|y_p| \leq R} |y_p|^{2\epsilon_{1, \nu}} \frac{\partial^{2\nu} \Phi_1}{\partial x_p^{\nu} \partial \bar{x}_p^\nu}(0, y_p; \sigma_{p, \nu})\, dy_p d\bar{y}_p \\
  + &\ \frac{-2 \pi i}{(\nu!)^2} \int_{|y_p| > R} |y_p|^{2\epsilon_{1, \nu}} \frac{\partial^{2\nu} \Phi_1}{\partial x_p^{\nu} \partial \bar{x}_p^\nu}(0, y_p; \sigma_{p, \nu})\, dy_p d\bar{y}_p, \quad R > 0,
\end{split}
\end{equation}
and analogously for the other chart \( V_p \).
\end{corollary}
\begin{proof}
The function \( \eta_1 \) can be chosen continuous and such that its restriction to \( E_p \) is \( \eta_1 \restr{E_p} \equiv 1 \) in \( |y_p| \leq R \) and \( 0 \) in \( |y_p| > R \). Because \( z_p y_p = 1 \) on the overlap of any two charts of \( E_p \), \( \eta_2\restr{E_p}\) must be identically \( 1 \) in \( |z_p| < 1/R \), i.e. \( |y_p| > R \), and zero in the complementary. The results follows now from the proof of the previous proposition. Indeed, substitute such \( \eta_1 \) and \( \eta_2 \) in \fref{eq:prop-residue-1} and use the fact that the integrand is a global section on \( E_p \).
\end{proof}

\begin{remark} \label{rmk:remark-integral}
The value of the residue \( \Res_{s = \sigma_{p, \nu}} \langle f^s, \varphi \rangle \) must be understood as the analytic continuation of the functions,
\[
\begin{split}
I_1(\alpha_1, \beta_3, \dots, \beta_r) & = \int_{\mathbb{C}} |y_p|^{2\alpha_1} \frac{\partial^{2\nu} \Phi_1}{\partial x_p^\nu \partial \bar{x}_p^\nu}(0, y_p; \beta_3, \dots, \beta_r)\, dy_p d\bar{y}_p, \\ I_2(\alpha_2, \beta_3, \dots, \beta_r) & = \int_\mathbb{C} |z_p|^{2 \alpha_2} \frac{\partial^{2\nu} \Phi_2}{\partial w_p^\nu \partial \bar{w}_p^\nu}(z_p, 0; \beta_3, \dots, \beta_r)\, dz_p d\bar{z}_p,
\end{split}
\]
at the rational points \( (\epsilon_{1, \nu}, \epsilon_{3, \nu}, \dots, \epsilon_{r, \nu}) \) and \((\epsilon_{2, \nu}, \epsilon_{3, \nu}, \dots, \epsilon_{r, \nu}) \), respectively, as these points will usually be outside the region of convergence of the integrals defining \( I_1, I_2 \).
 For simplicity of the notation, we always present \( \Phi_1, \Phi_2 \) depending only on a single variable \( \sigma_{p, \nu} \), as \( \epsilon_{3,r}, \dots, \epsilon_{r, \nu} \) are in fact \( N_3 \sigma_{p, \nu} + k_3 = \epsilon_{3, \nu}, \dots, N_r \sigma_{p, \nu} + k_r = \epsilon_{r, \nu} \).
\end{remark}


\vskip 2mm

Finally, we end this section with the following important observation. As in the monomial case \( \langle z^s, \varphi \rangle \) considered in \fref{prop:regularization}, where the residue is interpreted in terms of the derivatives of the test function \(\varphi\), and consequently, in terms of the Dirac's delta function, the same holds true for any \( f^s \). The derivatives of \( \Phi_1, \Phi_2 \) involve deriving \( \pi^*\varphi \) which, by the product rule of differentiation and the fact that \( (\pi^* \varphi)\restr{E_p} = \varphi(\boldsymbol{0})  \), imply that
\begin{equation} \label{eq:linear-combination-dirac}
\Res_{\ s = \sigma_{p, \nu}} f^s \in \big\langle \delta_{\boldsymbol{0}}^{(0,0,0,0)}, \delta_{\boldsymbol{0}}^{(1,0,0,0)}, \delta_{\boldsymbol{0}}^{(0,1,0,0)}, \dots, \delta_{\boldsymbol{0}}^{(\nu, \nu, \nu, \nu)} \big\rangle_{\mathbb{C}}.
\end{equation}
Therefore, the residue of \( f^s \) at any candidate pole must be also understood as a distribution in this precise sense. As a consequence, the residue of \( f^s \) at a candidate pole will be zero when all the coefficients of the linear combination in \fref{eq:linear-combination-dirac} are zero. In a similar way, the residue will be non-zero when just one of the coefficients is non-zero.

\subsection{Residues at non-rupture divisors} \label{sec:non-rupture-residues}
The exact expression of the residue is quite involved due to the presence of the \( \nu \)--th derivative of \( \Phi_1 \) or \( \Phi_2 \). However, from the study of the derivatives of the factors of \( \Phi_1 \) or \( \Phi_2 \) we will show when the residues at a candidate \( \sigma_{p, \nu} \) is zero for a non-rupture exceptional divisor \( E_p \). The proof uses the following technical results.

\begin{lemma}[Faà di Bruno's formula, {\cite[III.3.4]{comtet}}] \label{lemma:faa-di-bruno}
Let \( g, h \) be infinitely many times differentiable functions. Then,
\begin{equation} \label{eq:faa-di-bruno}
\frac{d^\nu}{d x^\nu} g(h(x)) = \sum_{k = 1}^\nu \frac{d^k g}{d x^k}(h(x)) B_{\nu, k}\left(\frac{dh}{dx}(x), \frac{d^2h}{dx^2}(x), \dots, \frac{d^{\nu -k +1}h}{dx^{\nu -k +1}}(x)\right),
\end{equation}
where \( B_{\nu, k} \) are the partial exponential Bell polynomials
\[ B_{\nu, k}(x_1, x_2, \dots, x_{\nu-k+1}) := \sum \frac{\nu!}{j_1! j_2! \cdots j_{\nu-k+1}!} \left(\frac{x_1}{1!}\right)^{j_1} \left(\frac{x_2}{2!} \right)^{j_2} \cdots \left(\frac{x_{\nu-k+1}}{(\nu - k + 1)!}\right)^{j_{\nu-k+1}}, \]
and where the summation takes places over all integers \( j_1, j_2, j_3, \dots, j_{\nu-k+1} \), such that
\begin{equation} \label{eq:bell-polynomials-2}
\begin{split}
& j_1 + j_2 + j_3 + \cdots + j_{\nu - k +1} = k, \\
& j_1 + 2j_2 + 3j_3 + \cdots + (\nu - k + 1)j_{\nu -k + 1} = \nu.
\end{split}
\end{equation}
\end{lemma}

For instance, in the chart \( U_p \) around \( E_p \), we are interested in the situation where \( g(x) = x^s \) and \( h(x_p) \) is equal to \(u_1(x_p, y_p)\) from \fref{prop:integral-along-divisor}, and we set \( x_p = 0 \) after deriving. In this case, \fref{eq:faa-di-bruno} reads as
\begin{equation} \label{eq:faa-di-bruno-2}
\frac{\partial^\nu u_1^s}{\partial x_p^\nu}(0, y_p) = \sum_{k = 1}^\nu (s)_k \big(u_1(0, y_p)\big)^{s-k} B_{\nu, k}\left(\frac{d {u_1}}{dx_p}(0, y_p), \frac{d^2 {u_1}}{dx_p^2}(0, y_p), \dots, \frac{d^{\nu - k + 1} {u_1}}{d x_p^{\nu - k + 1}}(0, y_p)\right),
\end{equation}
where \( (s)_k := s (s-1) \cdots (s-k+1) \). And similarly in the other chart \( V_p \).

\begin{proposition}[{\cite[I.3.8]{gelfand-shilov}}] \label{prop:residue-zero}
For any \( \alpha, \alpha' \in \mathbb{C} \) such that \( \alpha' - \alpha = n \in \mathbb{Z} \), the analytic continuation of the sum
\[ I_n(\alpha) := \int_{|z| \leq R} z^{\alpha'} \bar{z}^{\alpha} dz d\bar{z} + \int_{|z| > R} z^{\alpha'} \bar{z}^{\alpha} dz d\bar{z} \quad \textrm{for any} \quad R > 0. \]
is zero everywhere, i.e. \( I_n(\alpha) \equiv 0 \).
\end{proposition}
\begin{proof}
Using polar coordinates%
\footnote{By definition, \( z^{\alpha'} \bar{z}^\alpha := |z|^{\alpha' + \alpha} e^{i (\alpha' - \alpha) \arg z},\) which, for integral \( \alpha' - \alpha \), is a single valued function of \( z \).}
\[ - 2 i \int_{0}^R \int_{0}^{2\pi} r^{2\alpha + n + 1} e^{2 \pi i n \theta} d \theta dr - 2 i \int_{R}^\infty \int_{0}^{2 \pi} r^{2 \alpha + n + 1} e^{2 \pi i n \theta} d \theta dr. \]
However,
\[ \int_{0}^{2\pi} e^{2 \pi i n \theta} d \theta =  \begin{cases}
      \, 0,  & n \neq 0, \\
      \, 2 \pi, & n = 0.
\end{cases} \]
Hence, the result follows if \( n \neq 0 \). In the case that \( n = 0 \), the first integral defines an holomorphic function in \( \textrm{Re}(\alpha) > -1 \). It can be analytically continued by means of
\[ -4 \pi i \int_{0}^R r^{2 \alpha + 1} dr = - 2 \pi i \frac{R^{2(\alpha + 1)}}{\alpha + 1} \quad \textrm{for} \quad \alpha \neq -1. \]
Similarly, the other integral defines an holomorphic functions in \( \textrm{Re}(\alpha) < -1 \), and the analytic continuation to the whole complex plane is
\[ -4 \pi i \int_{R}^\infty r^{2\alpha + 1} dr = 2 \pi i \frac{R^{2(\alpha + 1)}}{\alpha + 1} \quad \textrm{for} \quad \alpha \neq -1. \]
Finally, the sum of the analytic continuation of both integrals is identically zero.
\end{proof}

In the following proposition, we generalize a calculation attributed to Cohen appearing in an article of Barlet \cite{barlet86-2} and used by Lichtin in \cite{lichtin89}. The original result gives a closed form for the integral in \fref{eq:residue-integral-1} in the case \( R_{0, 0}(\alpha, \beta) \). We provide a formula for the general case \( R_{n, m}(\alpha, \beta) \).

\begin{proposition} \label{prop:residue-integral}
For \( \alpha, \alpha', \beta, \beta' \in \mathbb{C} \), such that \( \alpha' - \alpha = n \in \mathbb{Z}\) and \( \beta' - \beta = m \in \mathbb{Z}\), the integral
\begin{equation} \label{eq:residue-integral-1}
R_{n, m}(\alpha, \beta) := \int_{\mathbb{C}} z^{\alpha'} \bar{z}^{\alpha} (1-\lambda z)^{\beta'}(1- \bar{\lambda}\bar{z})^{\beta} dz d \bar{z}  = R_{-n,-m}(\alpha', \beta'), \quad \lambda \in \mathbb{C}^*,
\end{equation}
is absolutely convergent for \( \emph{\textrm{Re}}(\alpha' + \alpha) > -2 \), \( \emph{\textrm{Re}}(\beta' + \beta) > -2 \) and \( \emph{\textrm{Re}}(\alpha' + \alpha + \beta' + \beta) < -2 \). It defines a meromorphic function on \(\mathbb{C}^2 \) equal to
\begin{equation} \label{eq:residue-integral-2}
R_{n,m}(\alpha, \beta) = - 2 \pi i \lambda^{-\alpha'-1} \bar{\lambda}^{-\alpha-1} \frac{\Gamma(\alpha + 1) \Gamma(\beta + 1) \Gamma(\gamma + 1)}{\Gamma(- \alpha - n) \Gamma(- \beta - m) \Gamma(- \gamma - n - m)},
\end{equation}
where \( \gamma := - \alpha - \beta - n - m - 2 \).
\end{proposition}
\begin{proof}
Let us prove first the case \( m = 0 \). Since \( R_{n, 0}(\alpha, \beta) = R_{-n, 0}(\alpha', \beta') \), we can assume \( n \in \mathbb{Z}_{\geq 0} \). Using polar coordinates \( \lambda z = re^{i\theta} \), we have that \( (1- \lambda z)^{\beta'}(1- \bar{\lambda} \bar{z})^{\beta} = |1-\lambda z|^{2\beta} = (1 - 2r \cos \theta + r^2)^{\beta} \) and
\[ R_{n,0}(\alpha, \beta) = -2i \lambda^{-\alpha'-1} \bar{\lambda}^{-\alpha-1} \int_0^{2\pi} \int_0^{\infty} r^{2\alpha + n + 1} e^{i n \theta} (1 - 2r \cos \theta + r^2)^{\beta} dr d\theta. \]
For simplicity, we can set \( \lambda = 1 \). We have, \( \displaystyle (1 - 2r \cos{\theta} + r^2)^{\beta} = (1+r^2)^{\beta} \left(1 - \frac{2r \cos{\theta}}{1 + r^2}\right)^\beta \) and since \( |2r/(1+r^2)| \leq 1 \), we may expand the binomial,
\[ \left(1 - \frac{2r \cos{\theta}}{1 + r^2}\right)^\beta = \sum_{k=0}^{\infty} \binom{\beta}{k} \frac{(-2)^k r^k}{(1+r^2)^k}\cos^k{\theta}. \]
The angular part of the integral is just
\[ \int_{0}^{2 \pi} e^{i n \theta} \cos^k{\theta}\ d \theta = \left \{
  \begin{array}{cl}
    \displaystyle \frac{2\pi}{4^l} \binom{2l}{l - s}, & \textrm{if} \quad k = 2l \geq n = 2s, \quad l \in \mathbb{Z}_{\geq 0}, s \in \mathbb{Z}_{\geq 0}, \vspace{0.2cm}\\
    \displaystyle \frac{\pi}{4^l} \binom{2l + 1}{l - s}, & \textrm{if} \quad k = 2l+1 \geq n = 2s+1, \quad l \in \mathbb{Z}_{\geq 0}, s \in \mathbb{Z}_{\geq 0}, \vspace{0.2cm}\\
    0, & \textrm{otherwise}.
  \end{array}
\right . \]
The integral then reads as
\[ R_{2s,0}(\alpha, \beta) = - 4 \pi i \int_{0}^{\infty} r^{2\alpha + 2s + 1} (1 + r^2)^\beta \sum_{l = 0}^\infty \binom{2l}{l-s} \binom{\beta}{2l} \frac{r^{2l}}{(1 + r^2)^{2l}} dr, \]
for \( n \) even, and
\[ R_{2s+1,0}(\alpha, \beta) = 4 \pi i \int_{0}^{\infty} r^{2\alpha + 2s + 2} (1 + r^2)^\beta \sum_{l = 0}^\infty \binom{2l + 1}{l-s} \binom{\beta}{2l + 1} \frac{r^{2l + 1}}{(1 + r^2)^{2l + 1}} dr, \]
for \( n \) odd. Using that \( \displaystyle \binom{\beta}{k} = (-1)^k \frac{\Gamma(k - \beta)}{\Gamma(- \beta)k!}\) and interchanging the summation and integral signs,
\[
\begin{split}
R_{2s,0}(\alpha, \beta) = \frac{-4 \pi i}{\Gamma(- \beta)} \sum_{l = 0}^{\infty} \frac{\Gamma(2l - \beta)}{(l-s)!(l+s)!} \int_{0}^{\infty} r^{2\alpha + 2s + 2l + 1} (1 + r^2)^{\beta - 2l} dr, \\
R_{2s+1,0}(\alpha, \beta) = \frac{-4 \pi i}{\Gamma(- \beta)} \sum_{l = 0}^{\infty} \frac{\Gamma(2l + 1 - \beta)}{(l-s)!(l+s+1)!} \int_{0}^{\infty} r^{2\alpha + 2s + 2l + 3} (1 + r^2)^{\beta - 2l - 1} dr.
\end{split}
\]
Now, for \( \textrm{Re}(\mu) > 0 \) and \( \textrm{Re}(2\nu + \mu) < 0\),
\[ \int_{0}^{\infty} x^{\mu - 1} (1 + x^2)^\nu dx = \frac{1}{2} \textrm{B}\left(\frac{\mu}{2}, -\nu - \frac{\mu}{2}\right) = \frac{1}{2}\Gamma\left(\displaystyle\frac{\mu}{2}\right) \Gamma\left(\displaystyle-\nu-\frac{\mu}{2}\right)\Gamma(-\nu)^{-1}. \]
See, for instance, \cite[3.251--2]{gradshteyn-ryzhik}. For \( \mu_l = 2(\alpha + s + l + 1) \) and \( \nu_l = \beta - 2l \),
\[ R_{2s}(\alpha, \beta) = \frac{-2\pi i}{\Gamma(- \beta)} \sum_{l = 0}^{\infty} \frac{\Gamma(\alpha + s + l + 1) \Gamma(l - \alpha - s - \beta - 1)}{(l-s)! (l+s)!}, \]
since \( \textrm{Re}(\mu_l) > 0 \) and \(\textrm{Re}(2\nu_l + \mu_l) < 0  \) for all \( l \in \mathbb{Z}_{\geq 0} \). Analogously, for \( n \geq 0 \) odd, \( \mu_l = 2(\alpha + s + l + 2) \) and \( \nu_l = \beta - 2l - 1 \),
\[ R_{2s+1,0}(\alpha, \beta) = \frac{-2\pi i}{\Gamma(- \beta)} \sum_{l = 0}^{\infty} \frac{\Gamma(\alpha + s + l + 2) \Gamma(l - \alpha - s - \beta - 1)}{(l-s)! (l+s+1)!}. \]
Since \( \Gamma(-k)^{-1} = 0 \) for \( k \in \mathbb{Z}_{\geq 0} \), we can write
\[ \begin{split}
R_{2s,0}(\alpha, \beta) = \frac{-2\pi i}{\Gamma(- \beta)} \sum_{l=0}^{\infty} \frac{\Gamma(\alpha + l + 1) \Gamma(l - \alpha - 2s - \beta -1)}{\Gamma(l - 2s + 1)\, l!}, \\
R_{2s+1,0}(\alpha, \beta) = \frac{-2\pi i}{\Gamma(- \beta)} \sum_{l=0}^{\infty} \frac{\Gamma(\alpha + l + 1) \Gamma(l - \alpha - 2s - \beta - 2)}{\Gamma(l - 2s)\, l!}.
\end{split} \]
Finally, for \(\textrm{Re}(c) > \textrm{Re}(a+b) \),
\begin{equation} \label{eq:hypergeometric-sum}
\sum_{k=0}^{\infty} \frac{\Gamma(a + k) \Gamma(b + k)}{\Gamma(c + k) k!} = \frac{\Gamma(a) \Gamma(b)}{\Gamma(c)} {}_{2}F_1(a, b; c; 1) = \frac{\Gamma(a)\Gamma(b)\Gamma(c - a - b)}{\Gamma(c - a) \Gamma(c - b)},
\end{equation}
where \( {}_2F_1(a, b; c; z) \) is the hypergeometric function. For the last equality see, for instance, \cite[9.122--1]{gradshteyn-ryzhik}. For \( n = 2s \), set \( a = \alpha + 1, b = -\alpha - \beta - 2s - 1\) and \(c = -2s + 1\). Similarly, for \( n = 2s + 1 \), \( a = \alpha + 1, b = -\alpha - \beta -2s -2 \) and \( c = -2s \). Then,
\[ R_{2s,0}(\alpha , \beta)\hspace{-0.5pt} = R_{2s+1,0}(\alpha, \beta) = \hspace{-0.5pt} -2 \pi i \frac{\Gamma(\alpha + 1) \Gamma(- \alpha - \beta - n - 1) \Gamma(\beta + 1)}{\Gamma(- \beta) \Gamma(-n - \alpha) \Gamma(\alpha + \beta + 2)}. \]
For the case where \( m \neq 0 \), having proved the result for \( n \in \mathbb{Z}, m = 0 \), we can assume without loss of generality that \( m \in \mathbb{Z}_{\geq 0} \). Keeping \( \lambda = 1 \), notice that
\[ R_{n,m}(\alpha, \beta) = \int_{\mathbb{C}} z^{\alpha'} \overbar{z}^{\alpha} |1-z|^{2 \beta}(1-z)^m dz d \overbar{z} =  \sum_{j=0}^m \binom{m}{j} (-1)^j R_{n+j,0}(\alpha, \beta). \]
Hence,
\[ R_{n, m}(\alpha, \beta) = -2 \pi i \frac{\Gamma(\alpha + 1) \Gamma(\beta + 1)}{\Gamma(- \beta) \Gamma(\alpha + \beta + 2)} \sum_{j = 0}^{m} \binom{m}{j} (-1)^j \frac{\Gamma(- \alpha - \beta - n - j - 1)}{\Gamma(-n -j -\alpha)}. \]
Since the binomial coefficient is zero if \( j > m \), we can consider the infinite sum. Expanding the binomial coefficient in terms of the Gamma function
\[ R_{n, m}(\alpha, \beta) = -2 \pi i \frac{\Gamma(\alpha + 1) \Gamma(\beta + 1)}{\Gamma(- \beta) \Gamma(\alpha + \beta + 2) \Gamma(-m) } \sum_{j = 0}^{\infty} \frac{\Gamma(j - m) \Gamma(- \alpha - \beta - n - j - 1)}{\Gamma(-n -j -\alpha)\, j!}. \]
Using the functional equation \( \Gamma(z+1) = z\Gamma(z) \) at each term
\[ \begin{split}
  R_{n, m}(\alpha, \beta) = & -2 \pi i \frac{\Gamma(\alpha + 1) \Gamma(\beta + 1) \Gamma(- \alpha - \beta -n-1) \Gamma(\alpha + \beta + n + 2)}{\Gamma(- \beta) \Gamma(\alpha + \beta + 2) \Gamma(-m) \Gamma(- \alpha - n) \Gamma(\alpha + n + 1)} \\ & \cdot \sum_{j=0}^\infty \frac{\Gamma(j - m) \Gamma(\alpha + n + j + 1)}{\Gamma(\alpha + \beta + n + j + 2) \, j!}.
  \end{split} \]
Applying \fref{eq:hypergeometric-sum} once again, since \( \textrm{Re}(\beta' + \beta) > -2 \) implies \( \textrm{Re}(\beta + m) > -1 \),
\[  R_{n, m}(\alpha, \beta) = -2 \pi i \frac{\Gamma(\alpha + 1) \Gamma(\beta + 1) \Gamma(- \alpha - \beta -n-1) \Gamma(\alpha + \beta + n + 2) \Gamma(\beta + m + 1)}{\Gamma(- \beta) \Gamma(\alpha + \beta + 2) \Gamma(- \alpha - n) \Gamma(\alpha + \beta + n + m + 2) \Gamma(\beta + 1) }. \]
And we get the desired result using the functional equation \( \Gamma(z+1) = z \Gamma(z) \) once again.
\end{proof}

It is now possible to prove the following result regarding non-rupture divisors. Recall that with the notation above, the divisors crossing a non-rupture exceptional divisor \( E_p \) can only be \( D_1, D_2, D_3 \), with at least one being non-zero and one being zero.

\begin{theorem} \label{thm:residue-non-rupture}
Let \( f :(\mathbb{C}^2, \boldsymbol{0}) \longrightarrow (\mathbb{C}, 0) \) be any plane branch. Let \( E_p, p \in K \) be a non-rupture exceptional divisor with sequence of candidate poles \( \sigma_{p, \nu}, \nu \in \mathbb{Z}_{\geq 0} \). Then,
\begin{itemize}
\item If \( D_3 = 0 \),
\[ \Res_{\ s = \sigma_{p, \nu}} f^s = 0, \quad \textrm{for all} \quad \nu \in \mathbb{Z}_{\geq 0}.\]
\item If \( D_3 \neq 0 \),
\[ \Res_{\ s = \sigma_{p, \nu}} f^s = 0, \quad \textrm{if} \quad \epsilon_{3, \nu} \not\in \mathbb{Z}.\]
\end{itemize}
\end{theorem}
\begin{proof}
Let us first begin with \( D_3 \) non-zero. In this case, we must also have \( D_1 \) or \( D_2 \) non-zero. We can assume, for instance, \( D_1 \) non-zero.

\vskip 2mm

By the definition of \( D_1, D_2, \dots, D_r \), if \( D_1 \) and \( D_3 \) are non-zero, \( D_3 \) is the only divisor crossing \( E_p \) in the \( V_p \) chart. In the \( (z_p, w_p) \) coordinates this means that \( u_2^s(z_p, 0) v_2(z_p, 0) \) has the form \( (1 - \lambda y_p)^{N_3 s + k_3} \) for some \( \lambda \in \mathbb{C}^* \). By Faà di Bruno's formula in \fref{lemma:faa-di-bruno} and \fref{eq:faa-di-bruno-2}, the \( \nu \)--th holomorphic and antiholomorphic derivatives of \( \Phi_2 \) at \( z_p, \bar{z}_p = 0, s = \sigma_{p, \nu} \) is an algebraic function involving the terms \( z_p^{k'} \bar{z}_p^{k} (1 - \lambda z_p)^{\epsilon_{3, \nu} - l'}(1 - \bar{\lambda} \bar{y}_p)^{\epsilon_{3, \nu} - l}\) with \( k, k', l,l' \in \mathbb{Z}_{\geq 0} \). Hence, the residue reduces to a finite sum involving the integrals \( R_{n, m}(\alpha, \beta) \) from \fref{eq:residue-integral-1} which, by \fref{prop:residue-integral}, are proportional to
\[ \frac{\Gamma(\alpha + 1) \Gamma(\beta + 1) \Gamma(\gamma + 1)}{\Gamma(- \alpha - n) \Gamma(- \beta - m) \Gamma(- \gamma - n - m)}, \quad n, m \in \mathbb{Z}. \]
Applying \fref{lemma:relation-epsilon} to this situation implies that, \( \epsilon_{1, \nu} + \epsilon_{3, \nu} + \kappa_p \nu + 2 = 0 \). Notice that \( \epsilon_{3, \nu} \not\in \mathbb{Z} \) implies \( \epsilon_{1, \nu} \not\in \mathbb{Z} \). Since the poles of \( \Gamma(z) \) are located at the negative integers, and \( \Gamma(z) \) has no zeros, \( \Gamma(\beta + 1), \Gamma(\gamma + 1), \Gamma(-\beta - m)^{-1}, \Gamma(-\gamma - n - m)^{-1} \in \mathbb{C}^* \). Finally, consider the non-negative integers \( k, k' \) in \( z_p^{k'} \bar{z}_p^{k} \) above. The quotient \( \Gamma(\alpha + 1)/\Gamma(- \alpha - n) \) is zero since \( \epsilon_{2, \nu} = 0 \), and \( \alpha + 1 = k + 1 > 0, -\alpha - n = -\alpha' = -k' \leq 0 \).

\vskip 2mm

If \( D_3 \) is zero and \( E_p \) is a non-rupture divisor, it may happen that \( D_2 \) is zero or not. If \( D_2 \) is not zero, it must cross \( E_p \) in the only point of \( V_p \) not in \( U_p \). In both cases, in the chart \( U_p \), none of the components \( u_1^s, v_1 \) of \( \Phi_1 \) will cross \( E_p \). In the \( (x_p, y_p) \) coordinates this means that, \( u_1^s(0, y_p) v_1(0, y_p) \in \mathbb{C}^*\) for all \( y_p \in \mathbb{C} \). Applying Faà di Bruno's formula again, the holomorphic and antiholomorphic derivatives of \( \Phi_1 \) restricted to \( x_p, \bar{x}_p = 0 \) are just polynomials in \( y_p \) and \( \bar{y}_p \). By \fref{cor:corollary-residue} and \fref{prop:residue-zero}, the residue is zero.
\end{proof}

\begin{example}
There are example where \( D_3 \neq 0 \) and \( \epsilon_{3, \nu} \in \mathbb{Z} \) and the corresponding \( \sigma_{p, \nu} \) has non-zero residue. For instance, \( f = (y^2 - x^3)(y - x^2)^3 \). Its resolution is given by
\[ F_\pi = 5E_{p_0} + 9E_{p_1} + 15E_{p_2} + C_1 + 3C_2, \quad K_\pi = E_{p_0} + 2E_{p_1} + 4E_{p_2}. \]
In this case, \( E_{p_1} \) is a non-rupture exceptional divisor with \( E_{p_2} \) and \( C_2 \) crossing \( E_{p_1} \). It can be checked that \( \sigma_{p_1, 0} = -1/3 \) is a pole of order two of \( f^s \). Here, \(D_1 = E_{p_2}, D_3 = C_2 \) and \( \epsilon_{1, 0} = \epsilon_{3, 0} = -1 \).
\end{example}

\section{The set of poles of the complex zeta function of a plane branch} \label{sec:poles-plane-branch}

In this section, we restrict our study of the poles of the complex zeta function to the case of plane branch singularities. Through the rest of this work we will fix a plane branch semigroup \( \Gamma = \langle \overline{\beta}_0, \overline{\beta}_1, \dots, \overline{\beta}_g \rangle \) and we stick to the notations of \fref{sec:semigroup}. Instead of taking any germ \( f : (\mathbb{C}^2, \boldsymbol{0}) \longrightarrow (\mathbb{C}, 0) \) with semigroup \( \Gamma \), we will work with the family \( \{f_{\boldsymbol{t}, \boldsymbol{\lambda}}\}_{\boldsymbol{\lambda} \in \mathbb{C}^{g-1}} \) presented in \fref{sec:monomial-curve}. Since this family contains at least one representative for each analytic type in the equisingularity class of the semigroup \( \Gamma \), we can give an optimal set of candidates for the poles of the complex zeta function \( f^s \) of all plane branches with semigroup \( \Gamma \). Furthermore, we will prove that if \( f_{gen} \) is generic among all branches with semigroup \( \Gamma \) (in the sense that the coefficients of a \( \mu \)-constant deformation are generic), all the candidates are indeed poles of \( f^s_{gen} \). As a corollary, we prove Yano's conjecture under the assumption that eigenvalues of the monodromy are pairwise different.

\subsection{Residues at rupture divisors} \label{sec:rupture-residues}
After \fref{thm:residue-non-rupture} in the preceding section, the only divisors that will contribute to poles of the complex zeta function of a plane branch will be rupture divisors and the divisor of the strict transform. Following the discussion in \fref{sec:toric-resolution}, a singular plane branch will have exactly \( g \geq 1 \) rupture divisors, denoted \( E_{p_1}, \dots, E_{p_g} \), where \( g \) is the number of characteristic exponents of \( f_{\boldsymbol{t}} \). This first observation reduces the list of candidate poles to \( \sigma_{i, \nu}, \nu \in \mathbb{Z} \) for \( i = 1, \dots, g \) contributed by \( E_{p_i} \), in addition to the negative integers corresponding to the strict transform \( \widetilde{C} \). With the notations from \fref{eq:numeric-data-semigroup},
\[ \sigma_{i, \nu} = - \frac{m_i + n_1 \cdots n_i + \nu}{n_i \overline{\beta}_i}, \quad \nu \in \mathbb{Z}_{\geq 0}. \]

We will basically use the same results about the residue presented in \fref{sec:residues-poles}. However, in this case, a better understanding of the total transform around the rupture divisors is needed. To that end, we will make use of \fref{prop:equations-curve}. The only thing that will differ from \fref{sec:residues-poles} is that, instead of computing the residue on the minimal resolution surface \( X' \), we proceed inductively and resolve up to the \( i \)--th rupture divisor, compute the residue on \( X^{(i)} \) and blow-up up to the \( (i+1) \)--th rupture divisor. This process simplifies the notation and fits more naturally with the toric resolutions from \fref{sec:toric-resolution}.

\vskip 2mm

First of all, we have to analyze the residue numbers from \fref{eq:definition-epsilons} in the case of a rupture divisor of a plane branch. In this case, \fref{eq:relation-epsilon} is \( \epsilon_{1, \nu} + \epsilon_{2, \nu} + \epsilon_{3} + \kappa_{p_i} \nu + 2 = 0 \), since rupture divisors of plane branches only have three divisor, \( D_1, D_2, D_3 \), crossing them. Since we will be working on \( X^{(i)} \), the divisor previously written as \( D_3 \) is the strict transform \( \widetilde{C} \) of \( f \) on \( X^{(i)} \) and thus, \(\epsilon_3 = e_i \sigma_{i, \nu} \), because \( E_{p_i} \cdot \widetilde{C} = e_i \). Similarly, on the surface \( X^{(i)} \), \( \kappa_{p_i} = - E_{p_i} \cdot E_{p_i} = 1 \). Hence,
\begin{equation}
\epsilon_{1, \nu} + \epsilon_{2, \nu} + e_i \sigma_{i, \nu} + \nu + 2 = 0.
\end{equation}
It is then enough to study the relation of \( \epsilon_{1, \nu}, \epsilon_{2, \nu}, \nu \in \mathbb{Z}_{\geq 0} \) with the semigroup \( \Gamma \). We point out that Lichtin \cite{lichtin85} studied the residue numbers \( \epsilon_{1, \nu}, \epsilon_{2, \nu} \) for the case \( \nu = 0 \). Using the notations and definitions from \fref{sec:semigroup},
\begin{proposition}[{\cite[Prop. 2.12]{lichtin85}}] \label{prop:lichtin}
The residue numbers \( \epsilon_{1, 0}, \epsilon_{2, 0} \) associated to a rupture divisor \( E_{p_i} \) of plane branch are given by
\[ \epsilon_{1, 0} + 1 = \frac{1}{n_i}, \quad \epsilon_{2, 0} + 1 = \frac{m_{i-1} - n_{i-1}\overline{m}_{i-1} + n_1 \cdots n_{i-1}}{\overline{m}_i}. \]
\end{proposition}
\begin{corollary} \label{cor:lichtin}
For any \( \nu \in \mathbb{Z}_{\geq 0} \), the residue numbers are
\[ \epsilon_{1, \nu} + 1 = -\frac{a_i}{n_i} \nu + \frac{1}{n_i}, \quad \epsilon_{2, \nu} + 1 = -\frac{c_i n_{i-1} \overline{m}_{i-1} + d_i}{\overline{m}_i} \nu + \frac{m_{i-1} - n_{i-1} \overline{m}_{i-1} + n_1 \cdots n_{i-1}}{\overline{m}_{i}}. \]
\end{corollary}
\begin{proof}
The result follows from the definition of \( \epsilon_{1, \nu}, \epsilon_{2, \nu} \) in \fref{eq:definition-epsilons} and using \fref{prop:lichtin} together with the expression of \( N_1 \) and \( N_2 \) given in \fref{prop:equations-curve}.
\end{proof}
Secondly, we focus our attention on the functions \( \Phi_1, \Phi_2 \) from \fref{prop:integral-along-divisor} but now in the case of the family \( f_{\boldsymbol{t}} \). Around the rupture divisor \( E_{p_i} \) on \( X^{(i)} \), we have that \( \Phi_{1, \boldsymbol{t}} = |u_1|^{2s} |v_1|^2 |\tilde{f}_{\boldsymbol{t}}|^{2s} (\pi^{(i)})^* \varphi \restr{U_i} \) and \( \Phi_{2, \boldsymbol{t}} = |u_2|^{2s} |v_2|^s |\tilde{f}_{\boldsymbol{t}}|^{2s} (\pi^{(i)})^* \varphi \restr{V_i} \).
Since the only factor of \( \Phi_{1, \boldsymbol{t}}, \Phi_{2, \boldsymbol{t}} \) crossing \( E_{p_i} \) is the strict transform \( \tilde{f}_{\boldsymbol{t}} \), by Faà di Bruno's formula in \fref{lemma:faa-di-bruno} applied to the equations in \fref{prop:equations-curve}, the \( \nu \)--th derivatives of \( \Phi_{1, \boldsymbol{t}}, \Phi_{2, \boldsymbol{t}} \) at \( x_p, \bar{x}_p = 0, w_p, \bar{w}_p = 0 \) with \( s = \sigma_{i, \nu} \) are a finite sum with summands that look like
\[ y_p^{k'} \bar{y}_p^{k} (y_i - \lambda_i)^{e_i(\sigma_{i, \nu} - l')} (\bar{y}_i - \bar{\lambda}_i)^{e_i(\sigma_{i, \nu} - l)}, \quad z_p^{k'} \bar{z}_p^{k} (1 - \lambda_i z_p)^{e_i(\sigma_{i, \nu} - l')} (1 - \bar{\lambda}_i \bar{z}_p)^{e_i(\sigma_{i, \nu} - l)}, \]
with \( k', k, l', l \in \mathbb{Z}_{\geq 0} \). Therefore, it make sense to consider the order and the degree of \( y_i, z_i \) in \( \Phi_{1, \boldsymbol{t}}, \Phi_{2, \boldsymbol{t}} \). On the \( U_i \) chart, they will be denoted by
\[
0 \leq  \ord_{y_i} \frac{\partial^{\nu} \Phi_{1, \boldsymbol{t}}}{\partial x^{\nu} \partial \bar{x}^\nu}(0, y_i; \sigma_{i, \nu}) \leq \deg_{y_i} \frac{\partial^{\nu} \Phi_{1, \boldsymbol{t}}}{\partial x^{\nu} \partial \bar{x}^\nu}(0, y_i; \sigma_{i, \nu}),
\]
and respectively on the \( V_i \) chart. By the symmetry of the holomorphic and antiholomorphic parts, the orders and degrees are exactly the same if considered with respect to the conjugated variables \( \bar{y}_i, \bar{z}_i \). By convention the order and degree of zero are \( + \infty \) and \( 0 \), respectively. Let us first present some technical lemmas.

\begin{lemma} \label{lemma:lower-bound-1}
Let \( f_1, \dots, f_n \in C^\infty(\mathbb{C}^2) \) be functions such that
\[ \alpha \nu \leq \displaystyle \ord_{y} \frac{\partial^\nu f_i}{\partial x^\nu}(0, y), \quad \deg_{y} \frac{\partial^\nu f_i}{\partial x^\nu}(0, y) \leq \beta \nu, \quad \textrm{for} \quad i = 1, \dots, n, \]
and \( \alpha, \beta \in \mathbb{Q}_{>0} \). Then,
\[ \alpha \nu \leq \ord_{y} \frac{\partial^\nu (f_1 \cdots f_n)}{\partial x^\nu}(0, y), \qquad \deg_{y} \frac{\partial^\nu (f_1 \cdots f_n)}{\partial x^\nu}(0, y) \leq \beta \nu,\]
and,
\[\alpha \nu \leq \ord_{y} \frac{\partial^\nu f_1^s}{\partial x^\nu}(0, y), \quad \deg_{y} \frac{\partial^\nu f_1^s}{\partial x^\nu}(0, y) \leq \beta \nu \quad \textrm{for all} \quad s \in \mathbb{C}. \]
\end{lemma}
\begin{proof}
The first inequalities follow from the general Leibniz rule,
\[ \frac{\partial^\nu (f_1 \cdots f_n)}{\partial x^\nu}(0, y) = \sum_{k_1 + \cdots + k_n = \nu} \binom{n}{k_1, \dots, k_n} \frac{\partial^{k_1} f_1}{\partial x^{k_1}}(0, y) \cdots \frac{\partial^{k_n} f_n}{\partial x^{k_n}}(0, y). \]
Similarly, the second follow from Faà di Bruno's formula and the definition of the partial exponential Bell polynomials, see \fref{lemma:faa-di-bruno} and \fref{eq:faa-di-bruno-2}. Specifically, they follow from the second part of \fref{eq:bell-polynomials-2}.
\end{proof}

\begin{lemma} \label{lemma:lower-bound-2}
Let \( f(x, y) \in C^\infty(\mathbb{C}^2) \) and let \( \pi(x, y) = (x^n y^a, x^m y^b) \) with \(n, m, a, b \in \mathbb{Z}_{\geq 0} \). Furthermore, assume that \( n b - m a \geq 0 \). Then,
\[ \frac{a}{n} \nu \leq \ord_{y} \frac{\partial^\nu ( f \circ \pi )}{\partial x^\nu} (0, y), \quad \deg_{y} \frac{\partial^\nu(f \circ \pi)}{\partial x^\nu}(0, y) \leq \frac{b}{m}\nu \quad \textrm{for all} \quad \nu \in \mathbb{Z}_{\geq 0}. \]
\end{lemma}
\begin{proof}
Consider the Taylor expansion of \( f \) at the origin of order \(\tau > \nu\),
\[ f(x, y) = \sum_{i, j = 0}^{\tau-1} \frac{\partial^{i+j} f}{\partial x^i \partial y^i}(0,0) \frac{x^i}{i!} \frac{y^j}{j!} + R_\tau(x, y) x^\tau y^\tau, \]
where \( R_\tau \) is the residual. Composing with \( \pi \), we get
\[ f(\pi(x, y)) = \sum_{i, j = 0}^{\tau-1} \frac{\partial^{i+j} f}{\partial x^i \partial y^i}(0,0) \frac{x^{ni+mj}}{i!} \frac{y^{ai + bj}}{j!} + R_\tau(x^ny^a, x^my^b) x^{(n+m)\tau} y^{(a+b)\tau}. \]
At the \( \nu \)--th derivative of this Taylor polynomial with respect to \( x \) and restricted to \( x = 0 \), we must have that \( n i + m j = \nu \) for some integers \( i, j \geq 0 \). Notice that if there are no such \( i, j \geq 0 \) the derivative is zero, and the bounds are trivially fulfilled. Finally,
\[ \begin{split}
\ord_{y} \frac{\partial^\nu (f \circ \pi)}{\partial x^\nu} (0, y) & = \min_{n i + m j = \nu} \{ ai + bj \} \geq \min_{j \geq 0} \left\{ \frac{a}{n}\nu + \frac{(nb - ma)j}{n} \right\} \geq \frac{a}{n} \nu, \\
\deg_{y} \frac{\partial^\nu(f \circ \pi)}{\partial x^\nu}(0, y) & = \max_{n i + m j = \nu} \{a i + b j\} \leq \max_{i \geq 0} \left\{ \frac{b}{m} \nu - \frac{(nb - ma)i}{m} \right\} \leq \frac{b}{m} \nu,
\end{split} \]
where in the first inequality of each equation used that \( i = (\nu - mj)/n \) and \( j = (\nu - ni)/m \), respectively.
\end{proof}

Recall the linear forms \( \rho_{j+1}^{(i)}(\underline{k}), A_{j+1}^{(i)}(\underline{k}), C_{j+1}^{(i)}(\underline{k}) \) from \fref{prop:equations-curve}.

\begin{lemma} \label{lemma:lower-bound-3}
For any \( \nu \in \mathbb{Z}_{\geq 0} \), we have
\[
\frac{a_i}{n_i}\nu \leq \hspace{-4pt} \min_{\rho_{j+1}^{(i)}(\underline{k}) = \nu} A^{(i)}_{j+1}(\underline{k}), \qquad \frac{c_i n_{i-1} \overline{m}_{i-1} + d_i}{\overline{m}_i} \nu \leq \hspace{-4pt} \min_{\rho_{j+1}^{(i)}(\underline{k}) = \nu} C^{(i)}_{j+1}(\underline{k}),
\]
\[
\max_{\rho_{j+1}^{(i)}(\underline{k}) = \nu} A^{(i)}_{j+1}(\underline{k}) \leq \frac{a_i n_{i-1} \overline{m}_{i-1} + b_i}{\overline{m}_i} \nu + n_{i+1} \cdots n_j, \quad \hspace{-4pt} \max_{\rho_{j+1}^{(i)}(\underline{k}) = \nu} C^{(i)}_{j+1}(\underline{k}) \leq \frac{c_i}{n_i} \nu + n_{i+1} \cdots n_{j}.
\]
\end{lemma}
\begin{proof}
For the first inequality, solve for \( k_0 \) in the constrain \( \rho_{j+1}^{(i)}(k_0, \dots, k_j) = \nu \) and substitute in \( A_{j+1}^{(i)}(\underline{k}) \). After some cancellations we get
\begin{equation} \label{eq:eqn1}
\frac{a_i}{n_i} \nu + a_i n_{i-1} \overline{m}_{i-1} k_i + b_i k_i - \frac{a_i}{n_i} \overline{m}_i k_i = \frac{a_i}{n_i}\nu + \frac{k_i}{n_i},
\end{equation}
where in the equality we applied \( \overline{m}_i = n_i n_{i-1} \overline{m}_{i-1} + q_i \), from \fref{eq:reduced-generators-semigroup}, and that \( n_i b_i - a_i q_i = 1 \). Since \( k_i \geq 0 \), we obtain the lower bound for the minimum of \( A_{j+1}^{(i)}(\underline{k}) \) restricted to \( \rho_{j+1}^{(i)}(\underline{k}) = \nu \). The argument for the lower bound for the minimum of \( C_{j+1}^{(i)}(\underline{k}) \) works similarly. Instead, solve for \( k_i \) in the constrain \( \rho_{j+1}^{(i)}(k_0, \dots, k_i, \dots, k_j) = \nu \) and substitute in \( C_{j+1}^{(i)}(\underline{k}) \). Applying \fref{eq:reduced-generators-semigroup} when necessary and that \( q_i c_i - n_i d_i = 1 \), we obtain
\begin{equation} \label{eq:eqn2}
\frac{c_i n_{i-1} \overline{m}_{i-1} + d_i}{\overline{m}_{i}} \nu + \frac{1}{\overline{m}_i} \sum_{l=0}^{i-1} n_{l+1} \cdots n_i \overline{m}_l k_l,
\end{equation}
which gives again the lower bound since \( k_l \geq 0 \). Having obtained the lower bounds for the minimums of \( A_{j+1}^{(i)} \) and \( C_{j+1}^{(i)} \) we can use \fref{lemma:lemma-ac1} to obtain the upper bounds for the maximums. Indeed,
\[ A_{j+1}^{(i)}(\underline{k}) + C_{j+1}^{(i)}(\underline{k}) \leq \rho_{j+1}^{(i)}(\underline{k}) + n_{i+1} \cdots n_{j}, \]
since \( k_l \geq 0 \). Hence,
\[ \max_{\rho_{j+1}^{(i)}(\underline{k}) = \nu} A^{(i)}_{j+1}(\underline{k}) \leq \nu + n_{i+1} \cdots n_j - \hspace{-4pt} \min_{\rho_{j+1}^{(i)}(\underline{k}) = \nu} C^{(i)}_{j+1}(\underline{k}) = \frac{a_i n_{i-1} \overline{m}_{i-1} + b_i}{\overline{m} _i} \nu + n_{i+1} \cdots n_j, \]
since \( a_i + c_i = n_i \) and \( b_i + d_i = q_i \). We can argue similarly for the remaining bound.
\end{proof}

All these technical lemmas are used in the proof of the following proposition.

\begin{proposition} \label{prop:lower-bound}
With the notations above, we have that
\[ \frac{a_i}{n_i}\nu \leq \ord_{y_i} \frac{\partial^{2\nu} \Phi_{1, \boldsymbol{t}}}{\partial x_i^\nu \partial \bar{x}^\nu_i}(0, y_i; \sigma_{i, \nu}), \quad \frac{d_i}{q_i}\nu \leq \ord_{z_i} \frac{\partial^{2\nu} \Phi_{2, \boldsymbol{t}}}{\partial w_i^\nu \partial \bar{w}_i^\nu}(z_i, 0; \sigma_{i, \nu}),\]
\[
\deg_{y_i} \frac{\partial^{2\nu} \Phi_{1, \boldsymbol{t}}}{\partial x_i^\nu \partial \bar{x}^\nu_i}(0, y_i; \sigma_{i, \nu}) \leq \frac{b_i}{q_i} \nu + e_i, \quad \deg_{z_i} \frac{\partial^{2\nu} \Phi_{2, \boldsymbol{t}}}{\partial w_i^\nu \partial \bar{w}_i^\nu}(z_i, 0; \sigma_{i, \nu}) \leq \frac{c_i}{n_i} \nu + e_i,
\]
for all \( \nu \in \mathbb{Z}_{\geq 0}\), and the same bounds hold for the conjugated variables \(\bar{y}_i, \bar{z}_i\).
\end{proposition}
\begin{proof}
We will do the proof for \( \Phi_{1, \boldsymbol{t}} \) since the proof for \( \Phi_{2, \boldsymbol{t}} \) works similarly. Recall that \( \Phi_{1, \boldsymbol{t}} = |u_1|^{2s} |v_1|^2 |\tilde{f}_{\boldsymbol{t}}|^{2s} (\pi^{(i)})^* \varphi\restr{U_i} \). By \fref{lemma:lower-bound-1}, it is enough to prove the bounds for each holomorphic or antiholomorphic factor. The factors \( u_1, v_1, (\pi^{(i)})^*\varphi\restr{U_i} \) are the pull-back by the toric morphism \( \pi_i \) of some invertible elements in the unique point of the total transform of \( f_{\boldsymbol{t}} \) that is not a simple normal crossing on \( X^{(i-1)} \). Therefore, by \fref{sec:toric-resolution} and \fref{lemma:lower-bound-2}, their orders (resp. degrees) with respect to \( y_i \) are bounded by \( a_i\nu/n_i \) (resp. \( b_i \nu/q_i \)). It remains to prove the bounds for the strict transform \( \tilde{f}_{\boldsymbol{t}} \).

\vskip 2mm

For the strict transform, consider \fref{prop:equations-curve} and proceed by induction from \( \tilde{f}_{i+1} \). Analyzing \fref{eq:prop-equations-curve-2}, the only part depending on \( x_i \) is the summation. By \fref{lemma:lower-bound-1}, it is enough to show that each factor of each summand fulfills the bounds. By the same argument as before, the units \( u_{\underline{k}}^{(i)} \) satisfy the bounds. On the other hand, \fref{lemma:lower-bound-3} assures the bounds for the monomials in \( x_i, y_i \). The lower-bound for the order is clear. For the upper-bound on the degree just notice that
\[ \frac{a_i n_{i-1} \overline{m}_{i-1} + b_i}{\overline{m}_i} = \frac{a_i n_{i-1} \overline{m}_{i-1} + b_i}{n_i n_{i-1} \overline{m}_{i-1} + q_i} < \frac{b_i}{q_i}, \]
since \( \overline{m}_i = n_i n_{i-1} \overline{m}_{i-1} + q_i \) and \( n_i b_i - q_i a_i = 1 \). Therefore, we are done for \( \tilde{f}_{i+1} \). By induction, if all \( \tilde{f}_{k+1}, i \leq k < j \) satisfy the bounds, so does \( \tilde{f}_{j+1} \). To see this, it is just a matter of applying \fref{lemma:lower-bound-1}, \fref{lemma:lower-bound-2}, \fref{lemma:lower-bound-3}, and the induction hypothesis to \fref{eq:prop-equations-curve-3}.
\end{proof}

We are ready to present the main result of this section.

\begin{theorem} \label{thm:plane-branch-candidates}
For any plane branch singularity \( f : (\mathbb{C}^2, \boldsymbol{0}) \longrightarrow (\mathbb{C}, 0) \) the poles of the complex zeta function \( f^s \) are simple and contained in the sets
\begin{equation} \label{eq:main-1}
\left\{ \sigma_{i, \nu} = - \frac{m_i + n_1 \cdots n_i + \nu}{n_i \overline{\beta}_i} \ \bigg|\ \nu \in \mathbb{Z}_{\geq 0},\   \overline{\beta}_i \sigma_{i, \nu}, e_{i-1} \sigma_{i, \nu} \not\in \mathbb{Z} \right\}, \quad i = 1, \dots, g,
\end{equation}
contributed by the rupture divisors \( E_{p_i} \), together with the negative integers \(\mathbb{Z}_{<0} \), contributed by the strict transform \( \widetilde{C} \).
\end{theorem}
\begin{proof}
In order to prove this result for any plane branch, it is enough to restrict the study to the family of \( f_{\boldsymbol{t}} \) from \fref{prop:monomial-curve-plane}. By the previous discussion, we only have to show that the candidates \( \sigma_{i, \nu} \) such that \( \overline{\beta}_i \sigma_{i, \nu}, e_{i-1} \sigma_{i, \nu} \in \mathbb{Z} \) have always residue zero. The first important observation is that \( \overline{\beta}_i \sigma_{i, \nu}, e_{i-1} \sigma_{i, \nu} \in \mathbb{Z} \), if and only if, \( \epsilon_{1, \nu}, \epsilon_{2, \nu} \in \mathbb{Z} \), respectively. To see this, consider the definitions \( \epsilon_{1, \nu} = N_1 \sigma_{1, \nu} + k_1, \epsilon_{2, \nu} = N_2 \sigma_{2, \nu} + k_2 \) from \fref{eq:definition-epsilons}.
Hence, \( \epsilon_{1, \nu}, \epsilon_{2, \nu} \in \mathbb{Z} \), if and only if, \( N_1 \sigma_{i, \nu}, N_2 \sigma_{i, \nu} \in \mathbb{Z} \), respectively. By \fref{prop:equations-curve}, \( N_1 = a_i \overline{\beta}_i \) and \( N_2 = (c_i n_{i-1} \overline{m}_{i-1} + d_i) e_{i-1} \), and the remark follows because
\[ \gcd(n_i \overline{\beta}_i, a_i \overline{\beta}_i) = \overline{\beta}_i \gcd(n_i, a_i) = \overline{\beta}_i, \]
since \( n_i b_i - q_i a_i = 1 \), and
\[ \gcd(n_i \overline{\beta}_i, (c_i n_{i-1} \overline{m}_{i-1} + d_i) e_{i-1}) = e_{i-1} \gcd(\overline{m}_i, c_i n_{i-1} \overline{m}_{i-1} + d_i) = e_{i-1}, \]
since \( \overline{m}_i = n_i n_{i-1} \overline{m}_{i-1} + q_i \) and \( q_i c_i - n_i d_i = 1 \). The argument to show that the residue is zero if \( \epsilon_{1, \nu} \in \mathbb{Z} \) or \( \epsilon_{2, \nu} \in \mathbb{Z} \) is fundamentally different for each case. Let us begin by the case where \( \overline{\beta}_i \sigma_{i, \nu} \in \mathbb{Z} \) and \( e_{i-1} \sigma_{i, \nu} \not\in \mathbb{Z} \).

\vskip 2mm

In order to study the residues of the candidates \( \sigma_{i, \nu} \) of the \( i \)--th rupture divisor \( E_{p_i} \) such that \( \overline{\beta}_i \sigma_{i, \nu} \in \mathbb{Z} \), we place ourselves in the chart \( U_i \) of \( X^{(i)} \) with local coordinates \( (x_i, y_i) \). The origin of this chart is the intersection point \( E_{p_1} \cap D_1 \). The only point of the total transform on \( X^{(i)} \) that is not a simple normal crossing is the intersection of the strict transform of \( f_{\boldsymbol{t}} \) with \( E_{p_i} \), with \( E_{p_i} \) being the only exceptional divisor at that point. Therefore, we can apply the formula for the residue in \fref{prop:residue} for the \( U_i \) chart.

\vskip 2mm

From the preceding discussion, the derivatives of \( \Phi_{1, \boldsymbol{t}} \) appearing in the residue formula are a finite sum of terms that look like \( y_p^{k'} \bar{y}_p^{k} (y_i - \lambda_i)^{e_i(\sigma_{i, \nu} - l')} (\bar{y}_i - \bar{\lambda}_i)^{e_i(\sigma_{i, \nu} - l)} \). Consequently, we can reduce the residue to a finite sum of the integrals from \fref{eq:residue-integral-1}, which, by \fref{prop:residue-integral}, are equal to
\begin{equation} \label{eq:main-theorem-1}
-2 \pi i \lambda^{-\alpha - n - 1} \bar{\lambda}^{-\alpha - 1} \frac{\Gamma(\alpha + 1) \Gamma(\beta + 1) \Gamma(\gamma + 1)}{\Gamma(- \alpha - n) \Gamma(- \beta - m) \Gamma(- \gamma - n - m)}, \quad n, m \in \mathbb{Z}.
\end{equation}
As noted earlier, for \( E_{p_i} \) on \( X^{(i)} \) we have that \( \epsilon_{1, \nu} + \epsilon_{2, \nu} + e_i \sigma_{i, \nu} + \nu + 2 = 0 \). If we are assuming that \( \epsilon_{1, \nu} \in \mathbb{Z} \) but \( \epsilon_{2, \nu} \not\in \mathbb{Z} \), i.e. \( e_{i-1} \sigma_{i, \nu} \not \in \mathbb{Z}  \), it must happen that \( e_i \sigma_{i, \nu} \not\in \mathbb{Z} \). This implies that \( \Gamma(\beta + 1), \Gamma(-\beta -n), \Gamma(\gamma + 1), \Gamma(-\gamma -n -m) \in \mathbb{C}^* \). However, the remaining factor, \( \Gamma(\alpha + 1) / \Gamma (-\alpha - n) \), is always zero, since \( \epsilon_{1, \nu} \in \mathbb{Z} \) implies that \( \alpha, \alpha' \in \mathbb{Z} \),
\[
\alpha + 1 = \epsilon_{1, \nu} + 1 + k \geq -\frac{a_i}{n_i} \nu + \frac{1}{n_i} + \frac{a_i}{n_i} \nu = \frac{1}{n_i} > 0, 
\]
and \( -\alpha - n = -\alpha' < 1 \), by \fref{cor:lichtin} and \fref{prop:lower-bound}. This proves that the residue at the candidate poles \( \sigma_{i, \nu} \) such that \( \overline{\beta}_i \sigma_{i, \nu} \in \mathbb{Z}, e_{i-1} \sigma_{i, \nu} \not\in \mathbb{Z} \) is zero.

\vskip 2mm

We move now to the case where \( e_{i-1} \sigma_{i, \nu} \in \mathbb{Z}, \overline{\beta}_i \sigma_{i, \nu} \not\in \mathbb{Z} \), i.e. \( \epsilon_{2, \nu} \in \mathbb{Z} \) and \( e_i \sigma_{i, \nu} \not\in \mathbb{Z} \). For this case observe that the previous argument, applied to the residue in the \( V_i \) chart, only works for the first rupture divisor. For \( i = 1 \), \( \epsilon_{2, \nu} + 1 = -{d_1}\nu/{q_1} + {1}/{q_1}, \) since \( n_{0}, m_0 = 0 \) and \( \overline{m}_1 = q_1 \). Otherwise, the formula for \( \epsilon_{2, \nu} \) from \fref{cor:lichtin} and the bound for the order of \( \Phi_{2, \boldsymbol{t}} \) in \fref{prop:lower-bound} do not match. Thus, from now on, we will assume \( i \geq 2 \).

\vskip 2mm

To study the residue at these poles, we consider \( \Res_{s = \sigma_{i, \nu}} f_{\boldsymbol{t}}^s \) as a function of \( \lambda_i \in \mathbb{C} \) and we place ourselves in the \( V_i \) chart. Since \( \lambda_i \) is the intersection coordinate of the strict transform with \( E_{p_i} \), if we let \( \lambda_i \rightarrow 0 \) or \( \lambda_i \rightarrow \infty \), we are in the situation of a non-rupture divisor and the residue is zero. By \fref{prop:residue-integral} and since \(\epsilon_2 \in \mathbb{Z} \), the residue is a Laurent series on \( \lambda_i, \bar{\lambda}_i \neq 0 \). Deriving under the integral sign in the formula for the residue from \fref{cor:corollary-residue} we increase the order in \( z_i, \bar{z}_i \) of \( \Phi_{2, \boldsymbol{t}} \) by one unit. Therefore, after deriving enough times we can assume that in \fref{eq:main-theorem-1}, \( \alpha + 1 > 0, - \alpha' \leq 0, \alpha, \alpha' \in \mathbb{Z} \). This implies that the principal part in \( \lambda_i, \bar{\lambda}_i \) of \( \Res_{s = \sigma_{i, \nu}} f^s_{\boldsymbol{t}}(\lambda_i) \) must be zero. However, the same argument is true if we consider the residue as a function of \( \lambda_i^{-1}, \overbar{\lambda}_i^{-1} \). Hence, the residue is independent of \(\lambda_i, \bar{\lambda}_i\). This implies that the residue is zero overall.

\vskip 2mm

It remains to show that the residue is zero in the case that \( \overline{\beta}_i \sigma_{i, \nu} \in \mathbb{Z} \) and \( e_{i-1} \sigma_{i, \nu} \in \mathbb{Z} \). Both conditions imply that, \( e_i \sigma_{i, \nu} \in \mathbb{Z} \). To see that the residue is zero in this situation, it is just a matter of combining the previous arguments and recalling that the Gamma function has only simple poles. After deriving with respect to \( \lambda_i \) in the \( V_i \) chart, we can get \fref{eq:main-theorem-1} with \( \alpha + 1 > 0 \) and \( - \alpha - n = -\alpha' \leq 0 \), i.e., the factor \( \Gamma(\alpha + 1) / \Gamma (- \alpha - n) \) is zero. Assume we have derived \( d' \) times with respect to \( \lambda_i \), then
\[ \alpha' = \epsilon_{2, \nu} + d' + k' \qquad \beta' = e_i \sigma_{i, \nu} - e_i l' - e_i d'. \]
Consequently, since \( e_i \geq 1 \),
\[
\begin{split}
\gamma + 1 & = -\alpha' - \beta' - 1 = - \epsilon_{2, \nu} + 1 - k' - e_i \sigma_{i, \nu} + e_i l' - 2 \geq \epsilon_{1, \nu} + \nu + 1 - k' + e_i \\
& \geq -\frac{a_i}{n_i} \nu + \frac{1}{n_i} + \nu - \frac{c_i}{n_i} \nu - e_i + e_i = \frac{1}{n_i} > 0,
\end{split}
\]
by \fref{cor:lichtin} and \fref{prop:lower-bound}. Similarly, \( - \gamma' - n - m < 1 \). Therefore, the factor \( \Gamma(\gamma + 1)/ \Gamma(-\gamma - n -m) \) is also zero. Since \( e_{i-1} \sigma_{i, \nu} \in \mathbb{Z} \), the piece, \( \Gamma(\beta + 1) / \Gamma(- \beta - m) \), has a pole. However, \fref{eq:main-theorem-1} is zero because the poles of \( \Gamma(z) \) are simple.

\vskip 2mm

Finally, we need to see that the negative integers, the candidates coming from the strict transform \( \widetilde{C} \), are poles. We can argue directly from the definition of \( f^s \) given in \fref{eq:definition-local-sing}. Take \( \boldsymbol{0} \neq p \in f^{-1}_{\boldsymbol{t}}(0) \cap U \) at which the equation \( f_{\boldsymbol{t}} \) can be taken as one of the holomorphic coordinates. Thus, we reduce the problem to the monomial case and, by \fref{prop:regularization}, the negative integers are simple poles. The poles contributed by rupture divisors are also simple because they must have \( \epsilon_{1, \nu}, \epsilon_{2, \nu}, e_i \sigma_{i,\nu} \not\in \mathbb{Z} \) for all \( \nu \in \mathbb{Z}_{\geq 0} \). These conditions imply that \fref{eq:main-theorem-1} cannot have a pole and hence, the residue does not have a pole. By \fref{sec:poles-order-two}, all the poles of \( f^s \) are simple.
\end{proof}

We point out that the candidate poles \( \sigma_{1, 0} > \sigma_{2, 0} > \cdots > \sigma_{g, 0} \) are always poles of \( f^s \) for any plane branch \( f \) as shown by Lichtin in \cite{lichtin85,lichtin89}.

\begin{example} \label{ex:example2}
There are examples where the candidate poles of \( f^s \) that are in the sets from \fref{eq:main-1} vary in a \( \mu \)--constant deformations of \( f \). For instance, consider \( f = y^4 - x^9 \) and the \( \mu \)--constant deformation \( f_t = y^4 - x^9 + tx^5y^2 \). For the unique rupture divisor, the sequence of candidate poles is \( \sigma_{1, \nu} = -(13 + \nu)/36, \nu \in \mathbb{Z}_{\geq 0} \). Taking \( \nu = 2 \), \( \sigma_{1, 2} = -5/12, \epsilon_{1, 2} = -9/4, \epsilon_{2, 2} = -4/3, \) and
\[ \Res_{s = \sigma_{1, 2}}f_t = -16 \pi^2 \sigma_{1, 2}^2 \frac{\Gamma(\epsilon_{1, 2} + 3) \Gamma(\sigma_{1, 2}) \Gamma(\epsilon_{2,2} + 2)}{\Gamma(-\epsilon_{1,2} - 2) \Gamma(-\sigma_{1,2} + 1) \Gamma(-\epsilon_{2, 2} - 1)} |t|^2 \delta_{\boldsymbol{0}}^{(0,0,0,0)}. \]
Therefore, \( \sigma_{1,2} = -5/12 \) is a pole, if and only if, \( t \neq 0 \).
\end{example}

\subsection{Generic poles}

Studying the residue in terms of the deformation parameters of \( f_{\boldsymbol{t}} \), we can get open conditions in which a certain candidate pole is indeed a pole, as seen in \fref{ex:example2}. Recalling \fref{eq:linear-combination-dirac}, the first observation is that, in terms of \( \boldsymbol{t} \),
\begin{equation} \label{eq:residue-deformation}
\Res_{s = \sigma_{i, \nu}} f_{\boldsymbol{t}}^s = \sum_{k',l',k,l = 0}^{\nu} p_{k',l'}(\boldsymbol{t}) p_{k,l}(\bar{\boldsymbol{t}})\, \delta^{(k',l',k,l)}_{\boldsymbol{0}},
\end{equation}
with \( p_{k', l'}(\boldsymbol{t}) = \overbar{p_{k, l}(\bar{\boldsymbol{t}})} \) if \( k' = k \) and \( l' = l \). The following theorem shows that, actually, any candidate is a pole in a certain Zariski open set in the deformation space of \( f_{\boldsymbol{t}} \).

\begin{theorem} \label{thm:generic}
For any \( M_1, \dots, M_g \in \mathbb{Z}_{\geq 0} \), generic plane branches \( f_{gen} \) in the equisingularity class corresponding to the semigroup \( \Gamma = \langle \overline{\beta}_0, \overline{\beta}_1, \dots, \overline{\beta}_g \rangle \) satisfy that
\begin{equation} \label{eq:theorem-generic-eq}
\left\{ \sigma_{i, \nu} = - \frac{m_i + n_1 \cdots n_i + \nu}{n_i \overline{\beta}_i} \ \bigg|\ 0 \leq \nu < M_i,\   \overline{\beta}_i \sigma_{i, \nu}, e_{i-1} \sigma_{i, \nu} \not\in \mathbb{Z} \right\}, \quad i = 1, \dots, g,
\end{equation}
are simple poles of the complex zeta function \( f^s_{gen} \).
\end{theorem}
\begin{proof}
By \fref{thm:plane-branch-candidates}, we only have to check the candidates in the sets given in \fref{eq:main-1}. In the case that \( \nu = 0 \), the residue for the candidates \( \sigma_{1, 0} > \sigma_{2, 0} > \cdots > \sigma_{g, 0} \) does not depend on \( \boldsymbol{t} \) and consists only of one of the integrals from \fref{prop:residue-integral} which is non-zero by \fref{prop:lichtin}.

\vskip 2mm

Therefore, fix a candidate pole \( \sigma_{i, \nu} \) with \( 0 < \nu < M_i \). It is enough to show that one of the polynomials \( p_{k',l'}(\boldsymbol{t}) \) on the parameters \( \boldsymbol{t} \) in the expression from \fref{eq:residue-deformation} is not identically zero. Consider the residue formula from \fref{prop:residue} in, for instance, the \( U_i \) chart and recall that \( \Phi_{1, \boldsymbol{t}} = |u_1|^{2s} |v_1|^2 |\widetilde{f}_{\boldsymbol{t}}|^{2s} (\pi^{(i)})^* \varphi \restr{U_i} \). Now, considering \fref{eq:prop-equations-curve-2}, we claim that the strict transform has a deformation term \[ t_{\underline{k}}^{(i)} x_i^{\rho_{i+1}^{(i)}(\underline{k})} y_i^{A_{i+1}^{(i)}(\underline{k})} u^{(i)}_{\underline{k}}(x_i, y_i) \] for a certain \( \underline{k} \) such that \( \rho_{i+1}^{(i)}(\underline{k}) = \nu \). Indeed, this happens since \( \rho_{i+1}^{(i)}(\underline{k}) = \nu \)
is equivalent, by \fref{eq:prop-equations-curve-4}, to
\[ n_1 \cdots n_i k_0 + n_2 \cdots n_i \overline{m}_1 k_1 + \cdots + \overline{m}_i k_i = n_i \overline{m}_i + \nu. \]
But this is an identity in the semigroup \( \Gamma_{i+1} \) of the maximal contact element \( f_{i+1} \), see \fref{eq:semigroup-max-contact}, and such a \( \underline{k} \) always exist because \( n_i \overline{m}_i + \nu \) is bigger than the conductor of \( \Gamma_{i+1} \). The deformation parameter \( t^{(i)}_{\underline{k}} \) for such a \( \underline{k} \) can only appear when \( f_{\boldsymbol{t}} \) is derived \(\nu\) times. By Faà di Bruno's formula, this implies that the polynomial \( p_{0,0}(\boldsymbol{t}) \) in \fref{eq:residue-deformation} is equal to \( \zeta t_{\underline{k}}^{(i)} + \cdots, \) where the dots represents other terms on \( \boldsymbol{t} \) not containing \( t_{\underline{k}}^{(i)} \). The coefficient \( \zeta \in \mathbb{C} \) is different from zero since it has the form of \fref{eq:residue-integral-2} and we are assuming that \( \overline{\beta}_i \sigma_{i, \nu}, e_{i-1} \sigma_{i, \nu} \not\in \mathbb{Z} \). Hence, the condition \( p_{0,0}(\boldsymbol{t}) \neq 0  \) gives a non-empty Zariski open subset on the deformation space of \( f_{\boldsymbol{t}} \) in which \( \sigma_{i, \nu} \) is a pole.

\vskip 2mm

In the case that there is a resonance between two poles, i.e. \( \sigma_{i, \nu} = \sigma_{i', \nu'}, i \neq i' \), we can always add, if necessary, an extra condition, giving a Zariski open set, which ensures that the residues do not cancel out. The intersection of all the open sets defines generic plane branches \( f_{gen} \).
\end{proof}
Consider the sets from \fref{eq:theorem-generic-eq} with \( M_i = n_i \overline{\beta}_i \), namely
\[
\Pi_i := \left\{ \sigma_{i, \nu} = - \frac{m_i + n_1 \cdots n_i + \nu}{n_i \overline{\beta}_i} \ \bigg|\ 0 \leq \nu < n_i \overline{\beta}_i,\   \overline{\beta}_i \sigma_{i, \nu}, e_{i-1} \sigma_{i, \nu} \not\in \mathbb{Z} \right\}, \]
for \( i = 1, \dots, g \) and define \( \Pi := \bigcup_{i=1}^g \Pi_{i} \). An easy computation shows that there are exactly \( \mu \) elements in \( \Pi \), counted with possible multiplicities,
\[ |\Pi| = \sum_{i=1}^g n_i \overline{\beta}_{i} - \overline{\beta}_i - n_i e_i + e_i = \sum_{i=1}^{g} (n_i - 1) \overline{\beta}_{i} + \sum_{i=1}^g e_i - e_{i-1} = \sum_{i=1}^g (n_i - 1) \overline{\beta}_{i} - n + 1 = \mu, \]
using \fref{eq:conductor2} in the last equality. The sets \( \Pi_i \) are precisely the \( b \)-exponents in Yano's conjecture from \fref{sec:yano-conjecture}. Indeed, the relation between the notations in \fref{eq:yano-definitions} and the resolution data in \fref{sec:toric-resolution} is clear. Namely, \( R_i = N_{p_i} = n_i \overline{\beta}_i, R'_i = N_{q_i} = \overline{\beta}_i, r_i = k_{p_i} + 1 = m_i + n_1 \cdots n_i \) and \( r'_i = k_{q_i} + 1 = \lceil (m_i + n_1 \cdots n_i)/n_i \rceil \). To see the equality between the exponents in \fref{eq:yano-generating-series} and the set \( \Pi \) is enough to notice that \( R_i = N_{p_i} = n_i N_{q_i} = n_i R'_i \) and \( r_i = k_{p_i} + 1 = n_i (k_{q_i} + 1) = n_i r'_i \).

\vskip 2mm

The results of Malgrange \cite{malgrange75} and Barlet \cite{barlet84} imply that the elements of \( \Pi \) generate all the eigenvalues of the monodromy. The characteristic polynomial of the monodromy is a topological invariant of the singularity, see A'Campo's formula \cite{acampo75}. Consequently, it can be computed from the semigroup \( \Gamma \) of \( f \), see \cite{neumann83}. The hypothesis that the eigenvalues of the monodromy are pairwise different is a condition on the equisingularity class, i.e. on the semigroup \( \Gamma \), implying that there are exactly \( \mu \) different elements in \( \Pi \).

\vskip 2mm

As a corollary of \fref{thm:generic}, we can deduce Yano's conjecture for any number of characteristic exponents if we assume that the eigenvalues of the monodromy are different. Yano's conjecture is proved by Cassou-Noguès for plane branches with one characteristic exponent in \cite{cassou-nogues88}. For two characteristic exponents with different monodromy eigenvalues, Yano's conjecture is proved by Artal Bartolo, Cassou-Noguès, Luengo and Melle Hernández in \cite{ABCNLMH16}.

\begin{corollary} \label{cor:conjecture-yano}
Let \( \Gamma = \langle \overline{\beta}_0, \overline{\beta}_1, \dots, \overline{\beta}_g \rangle \) be a semigroup defining an equisingularity class of plane branches. If the eigenvalues of the monodromy in the equisingularity class associated to the semigroup \( \Gamma \) are pairwise different, then Yano's conjecture holds.
\end{corollary}
\begin{proof}
We must check that all the \( \mu \) different elements of \( \Pi \) are roots of the Bernstein-Sato polynomial \( b_{{gen}, \boldsymbol{0}}(s) \) of \( f_{gen} \). If \( \sigma_{i, \nu} \in \Pi_i \) is in between \( - \textrm{lct}(f) - 1 \) and \( - \textrm{lct}(f) \), then \( \sigma_{i,\nu} \) is automatically a roots of \( b_{{gen}, \boldsymbol{0}}(s) \) by \fref{cor:shift-roots}. Otherwise, by \fref{cor:shift-roots}, we must check that \( \sigma_{i, \nu} + 1 \) is not a root of \( b_{{gen}, \boldsymbol{0}}(s) \). By contradiction, assume \( \sigma_{i, \nu} + 1 \) is a root of \( b_{{gen}, \boldsymbol{0}}(s) \). By \cite[Th. 1.9]{loeser85}, the roots of the Bernstein-Sato polynomial of a reduced plane curve can only be of the form \( \sigma - k \), with \( k \in \mathbb{Z}_{\geq 0} \) and \( \sigma \) a pole of \( f^s \). Hence, by \fref{thm:plane-branch-candidates}, \(\sigma_{i, \nu} + 1 = \sigma_{i', \nu'} \in \Pi_{i'}, i \neq i' \). But this is impossible since, as eigenvalues of the monodromy, they are both equal, contradicting the hypothesis. Finally, by the definitions in \fref{sec:yano-conjecture}, if the Bernstein-Sato polynomial has exactly \( \mu \) different roots they must coincide with the opposites in sign to the \( b \)-exponents.
\end{proof}

\bibliography{short,main}
\bibliographystyle{amsplain}

\end{document}

%% file: definitions.tex
\newcommand{\overbar}[1]{\mkern 1.5mu\overline{\mkern-3mu#1\mkern-0.5mu}\mkern 0.5mu}
\newcommand{\dd}{\text{d}}

\newcommand{\restr}[1]{\hspace{-0.3em}\mid_{#1}}

\DeclareMathOperator*{\Res}{Res}
\DeclareMathOperator{\ord}{ord}

\numberwithin{equation}{section}

\theoremstyle{plain}

\newtheorem{theorem}{Theorem}[section]
\newcommand*{\fancyrefthmlabelprefix}{thm}
\fancyrefaddcaptions{english}{
  \providecommand*{\frethmname}{Theorem}
  \providecommand*{\Frethmname}{Theorem}}
\frefformat{plain}{\fancyrefthmlabelprefix}{\frethmname\fancyrefdefaultspacing#1}
\Frefformat{plain}{\fancyrefthmlabelprefix}{\Frethmname\fancyrefdefaultspacing#1}

\newtheorem{proposition}[theorem]{Proposition}
\newcommand*{\fancyrefproplabelprefix}{prop}
\fancyrefaddcaptions{english}{
  \providecommand*{\frepropname}{Proposition}
  \providecommand*{\Frepropname}{Proposition}}
\frefformat{plain}{\fancyrefproplabelprefix}{\frepropname\fancyrefdefaultspacing#1}
\Frefformat{plain}{\fancyrefproplabelprefix}{\Frepropname\fancyrefdefaultspacing#1}

\newtheorem{lemma}[theorem]{Lemma}
\newcommand*{\fancyreflemmalabelprefix}{lemma}
\fancyrefaddcaptions{english}{
  \providecommand*{\frelemmaname}{Lemma}
  \providecommand*{\Frelemmaname}{Lemma}}
\frefformat{plain}{\fancyreflemmalabelprefix}{\frelemmaname\fancyrefdefaultspacing#1}
\Frefformat{plain}{\fancyreflemmalabelprefix}{\Frelemmaname\fancyrefdefaultspacing#1}

\newtheorem{corollary}[theorem]{Corollary}
\newcommand*{\fancyrefcorollarylabelprefix}{cor}
\fancyrefaddcaptions{english}{
  \providecommand*{\frecorollaryname}{Corollary}
  \providecommand*{\Frecorollaryname}{Corollary}}
\frefformat{plain}{\fancyrefcorollarylabelprefix}{\frecorollaryname\fancyrefdefaultspacing#1}
\Frefformat{plain}{\fancyrefcorollarylabelprefix}{\Frecorollaryname\fancyrefdefaultspacing#1}

\newtheorem*{conjecture}{Conjecture}
\newcommand*{\fancyrefconjecturelabelprefix}{ex}
\fancyrefaddcaptions{english}{
  \providecommand*{\freconjecturename}{Conjecture}
  \providecommand*{\Freconjecturename}{Conjecture}}
\frefformat{plain}{\fancyrefconjecturelabelprefix}{\freconjecturename\fancyrefdefaultspacing#1}
\Frefformat{plain}{\fancyrefconjecturelabelprefix}{\Freconjecturename\fancyrefdefaultspacing#1}

\theoremstyle{definition}

\newtheorem{definition}[theorem]{Definition}
\newcommand*{\fancyrefdeflabelprefix}{def}
\fancyrefaddcaptions{english}{
  \providecommand*{\fredefname}{Definition}
  \providecommand*{\Fredefname}{Definition}}
\frefformat{plain}{\fancyrefdeflabelprefix}{\fredefname\fancyrefdefaultspacing#1}
\Frefformat{plain}{\fancyrefdeflabelprefix}{\Fredefname\fancyrefdefaultspacing#1}

\newtheorem{remark}[theorem]{Remark}
\newcommand*{\fancyrefremarklabelprefix}{rmk}
\fancyrefaddcaptions{english}{
  \providecommand*{\freremarkname}{Remark}
  \providecommand*{\Freremarkname}{Remark}}
\frefformat{plain}{\fancyrefremarklabelprefix}{\freremarkname\fancyrefdefaultspacing#1}
\Frefformat{plain}{\fancyrefremarklabelprefix}{\Freremarkname\fancyrefdefaultspacing#1}

\newtheorem{example}{Example}
\newcommand*{\fancyrefexamplelabelprefix}{ex}
\fancyrefaddcaptions{english}{
  \providecommand*{\freexamplename}{Example}
  \providecommand*{\Freexamplename}{Example}}
\frefformat{plain}{\fancyrefexamplelabelprefix}{\freexamplename\fancyrefdefaultspacing#1}
\Frefformat{plain}{\fancyrefexamplelabelprefix}{\Freexamplename\fancyrefdefaultspacing#1}

\newcommand*{\fancyrefequationlabelprefix}{eq}
\fancyrefaddcaptions{english}{
  \providecommand*{\freequationname}{Equation}
  \providecommand*{\Freequationname}{Equation}}
\frefformat{plain}{\fancyrefequationlabelprefix}{\freequationname\fancyrefdefaultspacing(#1)}
\Frefformat{plain}{\fancyrefequationlabelprefix}{\Freequationname\fancyrefdefaultspacing(#1)}

\newcommand*{\fancyrefsectionlabelprefix}{sec}
\fancyrefaddcaptions{english}{
  \providecommand*{\fresectionname}{Section}
  \providecommand*{\Fresectionname}{Section}}
\frefformat{plain}{\fancyrefsectionlabelprefix}{\fresectionname\fancyrefdefaultspacing#1}
\Frefformat{plain}{\fancyrefsectionlabelprefix}{\Fresectionname\fancyrefdefaultspacing#1}

\newcommand*{\fancyrefalglabelprefix}{alg}
\fancyrefaddcaptions{english}{
  \providecommand*{\frealgname}{Algorithm}
  \providecommand*{\Frealgname}{Algorithm}}
\frefformat{plain}{\fancyrefalglabelprefix}{\frealgname\fancyrefdefaultspacing#1}
\Frefformat{plain}{\fancyrefalglabelprefix}{\Frealgname\fancyrefdefaultspacing#1}

\newcommand*{\fancyreffunclabelprefix}{func}
\fancyrefaddcaptions{english}{
  \providecommand*{\frefuncname}{Function}
  \providecommand*{\Frefuncname}{Function}}
\frefformat{plain}{\fancyreffunclabelprefix}{\frefuncname\fancyrefdefaultspacing#1}
\Frefformat{plain}{\fancyreffunclabelprefix}{\Frefuncname\fancyrefdefaultspacing#1}

\newcommand*{\fancyreflinelabelprefix}{line}
\fancyrefaddcaptions{english}{
  \providecommand*{\frelinename}{Line}
  \providecommand*{\Frelinename}{Line}}
\frefformat{plain}{\fancyreflinelabelprefix}{\frelinename\fancyrefdefaultspacing#1}
\Frefformat{plain}{\fancyreflinelabelprefix}{\Frelinename\fancyrefdefaultspacing#1}